\documentclass[11pt,a4paper,final,reqno]{amsart}
\usepackage[utf8]{inputenc}
\usepackage[T1]{fontenc}
\usepackage[UKenglish]{babel}
\usepackage[a4paper,margin=1in]{geometry}

\usepackage{amsmath,amsthm,amsfonts,amssymb}
\usepackage{mathrsfs,dsfont}
\usepackage{graphicx}
\usepackage{url}
\usepackage{verbatim}
\usepackage{xcolor}

\usepackage[autostyle]{csquotes} 

\usepackage{extarrows}

\renewcommand{\leq}{\leqslant}

\renewcommand{\geq}{\geqslant}

\renewcommand{\epsilon}{\varepsilon}

\newcommand{\real}{\mathds{R}}
\newcommand{\Comp}{\mathds{C}}
\newcommand{\Rd}{\real^{d}}
\newcommand{\nat}{\mathds{N}}
\newcommand{\Ecal}{\mathcal{E}}
\newcommand{\Dcal}{\mathcal{D}}
\newcommand{\sphere}{ \mathds{S}}
\newcommand{\Fourier}{\mathcal{F}}
\newcommand{\tinybullet}{{\scriptstyle\bullet}}

\newcommand{\loc}{\mathrm{loc}}

\newcommand{\casymp}[1]{\stackrel{#1}{\asymp}}

\newcommand{\I}{\mathds{1}}

\newcommand{\scalp}[2]{#1\cdot#2}
\newcommand{\nusj}{\mathring{\nu}}
\newcommand{\nulj}{\bar{\nu}}
\newcommand{\plj}{\bar{p}}
\newcommand{\psj}{\mathring{p}}

\theoremstyle{plain}
\newtheorem{theorem}{Theorem}[section]
\newtheorem{corollary}[theorem]{Corollary}
\newtheorem{lemma}[theorem]{Lemma}

\theoremstyle{definition}
\newtheorem{remark}[theorem]{Remark}
\newtheorem{example}[theorem]{Example}

\numberwithin{equation}{section}

\renewcommand\labelenumi{\textup{\alph{enumi})}}

\renewcommand\theenumi\labelenumi

\makeatletter\renewcommand{\p@enumii}{}\makeatother 

\makeatletter
\def\namedlabel#1#2{\begingroup
	#2%
	\def\@currentlabel{#2}%

	\label{#1}\endgroup
}

\newcounter{conum} 

\usepackage[unicode=true, pdfstartview={XYZ null null 1.00}, colorlinks=true, linkcolor=blue, urlcolor=blue, citecolor=purple]{hyperref}

\begin{document}

\title[Resolvent kernels and Schr\"odinger eigenstates for L\'evy operators]{Decay of resolvent kernels and Schr\"odinger eigenstates for L\'evy operators}
\author[K.~Kaleta]{Kamil Kaleta}
\address[K.~Kaleta]{
	Wroc\l aw University of Science and Technology, 
	Faculty of Pure and Applied Mathematics, 
	Wyb.\ Wyspia\'nskiego 27, 
	50-370 Wroc\l aw, Poland. 
	E-Mail: \textnormal{kamil.kaleta@pwr.edu.pl}}

\author[R.L.~Schilling]{Ren\'e L.\ Schilling}
\address[R.L.~Schilling]{
	TU Dresden, 
	Fakult\"at Mathematik, 
	Institut f\"ur Mathematische Stochastik, 
	01062 Dresden, Germany. 
	E-Mail: \textnormal{rene.schilling@tu-dresden.de}}

\author[P.~Sztonyk]{Pawe{\l} Sztonyk}
\address[P.~Sztonyk]{
	Wroc\l aw University of Science and Technology, 
	Faculty of Pure and Applied Mathematics, 
	Wyb.\ Wyspia\'nskiego 27, 
	50-370 Wroc\l aw, Poland. 
	E-Mail: \textnormal{Pawel.Sztonyk@pwr.wroc.pl}}



\thanks{
	Research was supported by National Science Centre, Poland, grant no.\ 2019/35/B/ST1/02421, the 6G-life project (BMBF 16KISK001K) and the Dresden--Leipzig ScaDS.AI centre. 
}

\begin{abstract}
We study the spatial decay behaviour of resolvent kernels for a large class of non-local Lévy operators and bound states of the corresponding Schrödinger operators. Our findings naturally lead us to proving results for Lévy measures, which have subexponential or exponential decay, respectively. This leads to sharp transitions in the the decay rates of the resolvent kernels. We obtain estimates that allow us to describe and understand the intricate decay behaviour of the resolvent kernels and the bound states in either regime, extending findings by Carmona, Masters and Simon for fractional Laplacians (the subexponential regime) and classical relativistic operators (the exponential regime). Our proofs are mainly based on methods from the theory of operator semigroups.
\end{abstract}

\maketitle

\section{Introduction and presentation of results}\label{intro}

In the seminal paper \cite{Carmona-Masters-Simon} Carmona, Masters and Simon studied the decay of the resolvent kernels $g_\alpha(x)=\int_0^\infty e^{-\alpha t}p_t(x)\, dt$, $\alpha > 0$. The density $p_t(x)$ is the transition density (heat kernel) of a L\'evy-type operator. The investigation of Carmona, Masters and Simon was mainly motivated by applications to decay properties of Schr\"odinger eigenfunctions. Among other results were the following findings:
\begin{itemize}
\item 
	Let $L_\beta u(x) := -(-\Delta)^{\beta/2}u(x) := \lim_{\epsilon\to 0}\int_{|y|>\epsilon} \left(u(x+y)-u(x)\right)\nu(y)\,dy$ with $\beta \in (0,2)$ and $\nu(y)=\nu_\beta(y) = c_\beta |y|^{-d-\beta}$, be the \textbf{fractional Laplacian}, then there exists some constant $c=c(\alpha)\geq 1$ such that for every fixed $\alpha>0$
	\begin{gather*}
		g_\alpha(x)\asymp \nu(x),\quad\text{i.e.,}\quad
		\frac 1c \nu(x) \leq g_{\alpha}(x) \leq c\nu(x), \quad |x| \geq 1.
	\end{gather*} 

\item 
	Let $L_mu(x) = \left(m - \sqrt{-\Delta+m^2}\right) u(x) = \lim_{\epsilon\to 0} \int_{|y|>\epsilon} \left(u(x+y)-u(x)\right) \nu(y)\,dy$, $m>0$, be the \textbf{relativistic operator}, where $\nu(y) = \nu_m(y) \asymp e^{-m|y|} \left(|y|^{-d-1} \vee |y|^{-(d+2)/2}\right)$ is an  exponentially localized L\'evy density. Then
	\begin{align} \label{eq:est_CMS}
		 g_{\alpha}(x) \asymp e^{-m_{\alpha}|x|} \times \text{polynomial terms},
	\quad |x| \geq 1, 
	\end{align}
	where
	\begin{gather*}
		m_{\alpha} := \begin{cases}
			\sqrt{2m \alpha - \alpha^2}, & \alpha \in (0,m],\\
			m,& \alpha > m;
		\end{cases}
	\end{gather*}
	in particular, it is no longer true that $g_{\alpha}(x)$ is comparable with $\nu(x)$, $|x| \geq 1$. 
\end{itemize}
The operators $L_\beta$ and $L_m$ are L\'evy operators, i.e.\ they are generators of convolution semigroups (or L\'evy processes), and it is a natural question to ask \emph{which features of a L\'evy operator $L$ -- or its L\'evy measure $\nu$ -- ensure that $g_{\alpha}(x)$ is comparable with $\nu(x)$, for all $\alpha >0$, and for large values of $x$?}

\bigskip
In this paper we investigate the spatial behaviour at infinity of the resolvent kernels of a large class of L\'evy operators, which appear naturally as the infinitesimal generators of convolution semigroups or L\'evy processes, extending the results of Carmona, Masters and Simon. Our main motivation and principal application is the precise description of the decay properties of the \textbf{bound states} (more precisely, the eigenfunctions corresponding to negative eigenvalues) of fairly general non-local Schr\"odinger operators. 

By $\{ \mu_t,\, t\geq 0 \}$ we denote a vaguely continuous \textbf{convolution semigroup} of probability measures, i.e.\ a family of probability measures on (the Borel sets of) $\Rd$ satisfying
\begin{gather*}
	\mu_{t+s} = \mu_t*\mu_s,
	\quad
	\mu_0 = \delta_0,
	\quad\text{and}\quad
	\lim_{t\to 0} \int\phi\,d\mu_t = \phi(0).
\end{gather*}
for all $s,t>0$ and all $\phi\in C_c(\Rd)$. It is straightforward to see that the operators $P_t\phi(x) := \int_{\Rd} \phi(x+y)\,\mu_t(dy)$, $t\geq 0$, define a strongly continuous \textbf{contraction semigroup}, both on $L^p(\Rd,dx)$, $1\leq p < \infty$, and $C_\infty(\Rd) := \overline{C_c(\Rd)}^{\|\cdot\|_\infty}$ (the continuous functions vanishing at infinity). Standard references are Berg \& Forst \cite{Berg-Forst} or Jacob \cite{Jacob-1}. Since $P_t$ is positivity preserving and conservative, there is a one-to-one correspondence between $\{ \mu_t,\, t\geq 0 \}$, $\{ P_t,\, t\geq 0 \}$, and the \textbf{Markov process} $\{ X_t,\, t\geq 0 \}$. This equivalence is realized by the relations $\mathrm{law}(X_t)=\mu_t$ or $P_tu(x) = \mathds{E}u(X_t+x)$. Because of the convolution structure, the Markov process is a \textbf{L\'evy process}, i.e.\ a stochastic process with independent and stationary increments, which is continuous in probability, cf.\ \cite{Jacob-1} or \cite{Schilling}.
	
We can characterize a convolution semigroup (or a L\'evy process) in terms of the Fourier transform, which is of the form 
$\Fourier(\mu_t)(\xi) = \int_{\Rd} e^{i\scalp{\xi}{y}}\,\mu_t(dy) = \exp(-t\Psi(\xi))$.
The function $\Psi : \Rd\to\Comp$ is the \textbf{characteristic exponent} which is given by the \textbf{L\'evy--Khintchine representation}
\begin{gather*}
	\Psi(\xi) 
	= i\scalp{\ell}{\xi} + \frac 12 \scalp{\Sigma\xi}{\xi} + \int_{\Rd\setminus\{0\}} \left(1- e^{i\scalp{\xi}{y}} + i\scalp{\xi}{y}\I_{(0,1)}(|y|) \right) \nu(dy),\quad \xi\in\Rd.
\end{gather*}
The triplet $(\ell,\Sigma,\nu)$ comprising a vector $\ell\in\Rd$, a positive semidefinite matrix $\Sigma\in\real^{d\times d}$ and a Radon measure $\nu$ on $\Rd\setminus\{0\}$ satisfying $\int_{\Rd} \left(1\wedge |y|^2\right) \nu(dy) < \infty$ is called \textbf{L\'evy triplet}, and the measure $\nu$ is called \textbf{L\'evy measure}. Again, there is a one-to-one correspondence between $\{\mu_t,\, t\geq 0\}$, $\Psi$ and the triplet $(\ell,\Sigma,\nu)$. In order to avoid trivial cases, we assume that $\nu(\Rd\setminus\{0\})=\infty$.

In this paper, we will assume that the measures $\mu_t$ are symmetric, i.e.\ $\mu_t(B)=\mu_t(-B)$ for any Borel set $B\subset\Rd$. This is equivalent to saying that \enquote{$\Psi$ is real-valued} or \enquote{$\ell=0$ and $\nu$ is symmetric}. We will also assume that $\Sigma=0$ and $\nu(dy)=\nu(y)\,dy$ is absolutely continuous w.r.t.\ Lebesgue measure. In this case the L\'evy--Khintchine representation becomes
\begin{gather}\label{Levy1}
	\Psi(\xi) 
	= \int_{\Rd\setminus\{0\}} \left(1-\cos(\scalp{\xi}{y}) \right) \nu(y)\,dy,\quad \xi\in\Rd.
\end{gather}

The characteristic exponent can be used to represent the infinitesimal generator $(L,\Dcal(L))$ of the operator semigroup $\{P_t, \, t\geq 0\}$ in, say, $L^2(\Rd)$. We have, see \cite{Jacob-1,Bottcher-Schilling-Wang}, 
\begin{gather} \label{def:gen}
	\Fourier[L u](\xi) = - \Psi(\xi) \Fourier u(\xi), \quad \xi \in \Rd, 
	\quad u \in \Dcal(L):=\left\{v \in L^2(\Rd): \Psi \Fourier v \in L^2(\Rd) \right\}.
\end{gather}
This means that $L$ is a (non-local) pseudo-differential operator with symbol $-\Psi$. The operator $L$ is also known as \textbf{L\'evy operator}.

In order to control the symbol $\Psi$, the following maximal function $\Psi^*$ and its generalized inverse $\Psi^*_{-}$ are useful: 
\begin{gather}\label{eq:maximal}
	\Psi^*(r)=\sup_{|\xi|\leq r}  \Psi(\xi),\quad r \geq 0;
\end{gather}
$\Psi^*$ is continuous, increasing and satisfies $\lim_{r \to \infty} \Psi^*(r)=\infty$. Its generalized inverse function is defined as 
\begin{align}\label{eq:maximal_inv}
	\Psi^*_{-}(s)=\sup\{r>0: \Psi^*(r)=s\}, \quad s>0.
\end{align}
By definition,
\begin{gather*}
	\Psi^*(\Psi^*_{-}(s))=s 
	\quad\text{and}\quad 
	\Psi^*_{-}(\Psi^*(s))\geq s, \quad s>0.
\end{gather*}

\bigskip\noindent
Throughout this paper we use two basic assumptions \eqref{L1} and \eqref{L2}.
\begin{itemize}  
	\item[\bfseries(\namedlabel{L1}{L1})] 
	There exists a constant $C_1>0$ such that
	\begin{gather*}
		\int_{\Rd} e^{-t\Psi(\xi)}\, d\xi \leq C_1 \left(\Psi^*_{-}\left(\tfrac 1t\right)\right)^d,
		\quad t>0.
	\end{gather*}
\end{itemize}
The condition \eqref{L1} ensures that the measures $\mu_t$ are for every $t>0$ absolutely continuous w.r.t.\ Lebesgue measure, $\mu_t(dy) = p_t(y)\,dy$, with a bounded and continuous transition density $p_t$, cf.\ \cite{Knopova-Schilling-2013}. 
Since
\begin{align} \label{eq:sup_pt}
	p_t(x) \leq p_t(0) = (2 \pi)^{-d} \int_{\Rd} e^{-t\Psi(\xi)}\, d\xi, \quad x \in \Rd, \;t>0,
\end{align}
it provides a good control on the suprema of densities. It follows from \cite[Theorem 3.1]{GrzywnySzczypkowski2020} that \eqref{L1} is equivalent to the property that 
\begin{align}\label{eq:scaling}
	\exists \, C_2 \in (0,1),\; \alpha \in (0,2] \::\: \Psi^*(\lambda r)\geq C_2 \lambda^{\alpha}\Psi^*(r),\; \lambda \geq 1,\; r>0.
\end{align}

\medskip
Although $\nu(y)$ is symmetric, it need not be rotationally symmetric. The following assumption will allow us to control $\nu(y)$ by a rotationally symmetric function:\ we assume that there exists a decreasing function $f: (0,\infty) \to (0,\infty)$ and a constant $C_0>0$ such that
\begin{gather*}
	C_0^{-1} f(|x|) \leq \nu(x) \leq C_0 f(|x|), \quad x\in\Rd\setminus \{0\}.
\end{gather*}
We call $f$ the \textbf{profile} of the L\'evy density $\nu$. Profiles are frequently used in the literature to control the behaviour of the L\'evy measure, see e.g.\ \cite{CKW,Finkelshtein,KSV} and further references quoted in this section. 

\bigskip\noindent
We can now formulate the second key assumption:
\begin{itemize}
	\item[\bfseries(\namedlabel{L2}{L2})] 
	The density $\nu$ admits a profile $f$, and we have $K_f(r)\to 0$ as $r \to \infty$, where
	\begin{gather*}
		K_f(r):=\sup_{|x|\geq 1} \frac{\displaystyle\int_{{|y-x|> r,\: |y| > r}} f(|x-y|)f(|y|)\, dy}{f(|x|)}, \quad r \geq 1.
	\end{gather*}
\end{itemize}

\noindent
The function $K_f$ is an important tool in the study of the long-range properties of pure jump L\'evy processes and the corresponding convolution semigroups. It is known that $K_f(1)<\infty$ is a necessary and sufficient condition for the uniform comparability $p_t(x) \asymp t f(|x|)$ for large $x$ and small $t$, see \cite{KaletaSztonyk2017}, while \eqref{L2} describes the directional spatial convergence of $p_t(x)/(t\nu(x))$ as $|x| \to \infty$ \cite{KaletaSztonyk2019, KaletaPonikowski2022}. Typical examples of profiles $f$ satisfying \eqref{L2} can be found in \cite[Lemma 8]{KaletaSztonyk2019}, see also Example \ref{ex:ex1} below.  Some easy-to-check sharp sufficient conditions for \eqref{L2} can also be obtained by an obvious modification of \cite[Lemma 3.2]{KaletaSchilling}. We remark that this condition rules out profile functions $f$ with super-exponential decay at infinity. 

\medskip
The following two-sided estimate, which follows from \cite[Theorem 3]{KaletaSztonyk2017}, will be frequently used: under \eqref{L1} and \eqref{L2} there exists a constant $C_3>0$ such that
\begin{align} \label{eq:small_time}
	C_3^{-1} t f(|x|) \leq p_t(x) 
	\leq C_3 t f(|x|), \quad |x| \geq \frac 12,\; t\in (0,1].
\end{align} 

Our first main result, Theorem \ref{th:subexp}, gives necessary and sufficient conditions answering the question on the asymptotic behaviour of $g_\alpha(x)$ if $|x|\gg 1$. We assume that \eqref{L1} and \eqref{L2} hold. Because of \eqref{eq:small_time} we always have
\begin{align} \label{eq:res_low}
	g_{\alpha}(x) \geq c f(|x|), \quad |x| \geq 1,\; \alpha >0,
\end{align}
with $c=c(\alpha)$, so it is enough to consider the upper bound for $g_{\alpha}(x)$.

\begin{theorem}\label{th:subexp} 
	If \eqref{L1} and \eqref{L2} hold, then the following statements are equivalent.
	\begin{enumerate}
	\item\label{th:subexp-A} 
	The profile $f$ is sub-exponential, i.e.\     
	\begin{equation}\label{eq:sub_exp} 
		\lim_{r\to \infty} \frac{\log f(r)}{r} = 0.
	\end{equation}
	
	\item\label{th:subexp-B}
	For every $\epsilon>0$ there exists a constant $\tilde{C}=\tilde{C}(\epsilon)>0$ such that 
	\begin{equation}\label{eq:doubling_kappaB}
		f(r) \geq \tilde{C} e^{-\epsilon r},\quad r\geq 1.
	\end{equation}
	
	\item\label{th:subexp-C}
	For every $\alpha_0>0$ there exists a constant $C= C(\alpha_0)>0$ such that 
	\begin{equation}\label{eq:doubling_kappaC}
		p_t(x) \leq C e^{\alpha_0 t} f(|x|),\quad |x|\geq 1,\; t>0.
	\end{equation}
	
	\item\label{th:subexp-D}
	For every $\alpha_0>0$ there exists a constant $C= C(\alpha_0)>0$ such that for every $|x|\geq 1$ the function 
	\begin{equation}\label{eq:doubling_kappaD}
		\alpha \mapsto \frac{C}{\alpha - \alpha_0} f(|x|) - g_\alpha(x)
	\end{equation}
	is completely monotone on $(\alpha_0,\infty)$. 

	\item\label{th:subexp-E}
	For every $\alpha_0>0$ there exists a constant $C= C(\alpha_0)>0$ such that for every $\alpha>\alpha_0$ we have
	\begin{equation}\label{eq:doubling_kappaE}
		g_\alpha(x) \leq \frac{C}{\alpha-\alpha_0} f(|x|), \quad |x|\geq 1.
	\end{equation}
	\end{enumerate}
\end{theorem}
The condition \ref{th:subexp-C} in Theorem~\ref{th:subexp} indicates that the decay of $g_\alpha(x)$, $|x|\to\infty$, depends on the growth/decay of the heat kernel $p_t(x)$ in both $t$ and $x$. 

Theorem~\ref{th:subexp} settles the case of subexponential profiles $f$. Our next goal is to understand the behaviour of $g_{\alpha}(x)$ for exponentially decaying L\'evy measures. We consider profile functions that combine an exponential and a lower-order term, i.e. 
\begin{gather} \label{eq:exp_prof}
	f(r) := \exp(-\kappa r) h(r), \quad r > 0,
\intertext{where $\kappa>0$ and $h:(0,\infty) \to (0,\infty)$ is a decreasing function such that the map} 
\label{eq:g_sub_exp}
	r \mapsto \frac{\log h(r)}{r} 
	\text{\ \ is eventually increasing and\ \ }
	\lim_{r\to \infty} \frac{\log h(r)}{r} = 0.
\end{gather}
This means that $h$ is sub-exponential in the sense of \ref{th:subexp}.\ref{th:subexp-A}.\footnote{By \ref{th:subexp}.\ref{th:subexp-A}, \ref{th:subexp}.\ref{th:subexp-B} etc.\ we refer to the statements \ref{th:subexp-A}, \ref{th:subexp-B} etc.\ of Theorem \ref{th:subexp}.} We note that the relativistic case mentioned above belongs in this framework.

Let $\xi \in \Rd$ be such that the following integral is finite:
\begin{gather*}
	\omega(\xi) = \int_{\Rd\setminus\{0\}} \left(\cosh (\scalp{\xi}{y}) - 1\right) \nu(dy).
\end{gather*}
Using the elementary relation $\cosh(i\theta) = \cos(\theta)$ it is easy to see that $\omega(\xi) = - \Psi(\frac 1i \xi)$ holds for all $\xi \in \Rd$ such that $\Psi$ can be (analytically) extended onto the strip $\Rd + iU$, where $U = \left\{\eta \in \Rd: \int_{|y|>1} e^{\eta \cdot y}\, \nu(dy) < \infty\right\}$. In fact, the moment condition appearing in the definition of $U$ is a necessary and sufficient condition for the extendability of $\Psi$ into $\Comp^d$, see the discussion in Berger \emph{et al.}~\cite[Section 4]{Berger-Schilling-Shargorodsky}. This is also the condition that ensures the finiteness of $\omega(\xi)$. Lemma \ref{lem:exp_mom} below shows that, in the present setting, where the profile of the L\'evy density satisfies \eqref{eq:exp_prof}--\eqref{eq:g_sub_exp}, the assumption \eqref{L2} allows us to take $U = \left\{\eta \in \Rd: |\eta| \leq \kappa \right\}$. 

For a fixed $\theta \in \sphere^{d-1}$ we consider the function $s\mapsto \omega(s\theta)$. It is easy to see that $s\mapsto\omega(s\theta)$ is strictly convex on the set $[-\kappa,\kappa]$ 
with a unique minimum at $s=0$, hence it is strictly decreasing on $(-\kappa,0)$ and strictly increasing on $(0,\kappa)$. Therefore, the following inverse function is well-defined: 
\begin{gather*}
\gamma_{\alpha}(\theta) := 
\begin{cases}
	\kappa,\ & \alpha > \omega(\kappa \theta),\\
	\omega(\tinybullet\theta)^{-1}(\alpha),\ & 0 < \alpha \leq \omega(\kappa \theta).
\end{cases}
\end{gather*} 
It is clear that $\alpha \mapsto \gamma_{\alpha}(\theta)$ is a continuous function. For a radial L\'evy measure $\nu$, the function $\omega$ is again radial, so that $\gamma_{\alpha}(\theta)\equiv \gamma_{\alpha}$ does not depend on $\theta \in \sphere^{d-1}$.

In Theorem \ref{th:exp} below, our second main result, we obtain exponential estimates: we construct the directional upper bound for general densities, and two-sided estimates for radial densities. 

\begin{theorem}[sharp transition in the exponential rate] \label{th:exp} 
	Let the profile $f$ be of the form \eqref{eq:exp_prof}--\eqref{eq:g_sub_exp}, and assume \eqref{L1} and \eqref{L2}. Then the following assertions hold. 
	\begin{enumerate}
	\item\label{th:exp-1} 
	\textup{(Directional upper bounds for general L\'evy measures)}
	There exists a constant $C>0$ such that for every $\alpha_0 > 0$ we have
	\begin{gather*}
		p_t(r\theta) 
		\leq C e^{\alpha_0 t} \exp\left(-\gamma_{\alpha_0}(\theta)\, r\right), 
		\quad r \geq 1,\; \theta \in \sphere^{d-1},\; t>0,
	\intertext{and}
		g_{\alpha}(r\theta) 
		\leq \frac{C}{\alpha-\alpha_0} \exp\left(-\gamma_{\alpha_0}(\theta)\, r\right), 
		\quad r \geq 1,\; \theta \in \sphere^{d-1},\; \alpha > \alpha_0.
	\end{gather*} 
	
	\item\label{th:exp-2} 
	\textup{(Two-sided bound for radial L\'evy measures)} 
	Assume, in addition, that the L\'evy measure $\nu$ is radial, i.e.\ $\nu(x)=\nu(|x|)$. Let $\alpha >0$. Then for every $\epsilon >0$ there exists a constant $\widetilde C= \widetilde C(\epsilon,\alpha)>1$ such that 
	\begin{gather*}
		\widetilde C^{-1} e^{-(\gamma_{\alpha}+\epsilon)|x|} 
		\leq g_{\alpha}(x) \leq \widetilde C e^{-(\gamma_{\alpha}-\epsilon)|x|}, 
		\quad |x| \geq 1.
	\end{gather*}
	In particular, the transition in the decay rate depending on the position of $\alpha$ with respect to $\omega^*(\kappa):= \sup_{\theta \in \sphere^{d-1}} \omega(\kappa \theta)$ is sharp.\footnote{$\omega^*(\kappa)$ is well-defined both in the radial and non-radial setting. In the radial setting $\omega(\kappa \theta)$ does not depend on $\theta \in \sphere^{d-1}$.}
	\end{enumerate}
\end{theorem}
Theorem \ref{th:exp}.\ref{th:exp-2} allows us to understand the intricate structure of the decay rate in the exponential case; in particular, there is a sharp continuous transition, depending on the position of $\alpha$ with respect to some explicit threshold parameter $\omega^*(\kappa)$ (this is the role of $m$ in \eqref{eq:est_CMS}). In Section \ref{sec:applications} we provide a general interpretation of $\omega^*(\kappa)$, which will help to understand what is going on in \eqref{eq:est_CMS}.

The proof of the above result uses in an essential way an observation which is stated as Theorem~\ref{th:gen_exp}. This theorem gives a general upper bound for the density of a symmetric convolution semigroup that has \textbf{some} finite exponential moment. It extends the result by Knopova and Schilling \cite[Theorem 6]{KnopovaSchilling2012}, which provides such a bound for measures having \textbf{all} exponential moments finite. In other words, we extend this bound to exponentially localized L\'evy measures. 

Our third main result improves the estimates in Theorem \ref{th:exp}.\ref{th:exp-1} if $\alpha_0 > \omega^*(\kappa)$ and the two-sided bound in Theorem \ref{th:exp}.\ref{th:exp-2} if $\alpha > \omega^*(\kappa)$.

\begin{theorem}[sharp radial estimates above threshold for general L\'evy measures]\label{th:exp_sharp} 
	Let the profile $f$ be of the form \eqref{eq:exp_prof}--\eqref{eq:g_sub_exp}, and assume \eqref{L1} and \eqref{L2}. For every $\alpha_0 > \omega^*(\kappa)$ there exist constants $C=C(\alpha_0)>0$ and $\rho=\rho(\alpha_0)$ such that
	\begin{gather*}
		p_t(x) \leq C e^{\alpha_0 t} f(|x|), \quad |x| \geq \rho,\; t>0,
	\intertext{and}
		g_{\alpha}(x) \leq \frac{C}{\alpha-\alpha_0} f(|x|), \quad |x|\geq \rho,\; \alpha > \alpha_0.
	\end{gather*} 
	In particular, for every $\alpha > \omega^*(\kappa)$ there exists $\widetilde C= \widetilde C(\alpha)$ such that $g_{\alpha}(x) \casymp{\widetilde C} f(|x|)$ for $|x|\geq \rho$.
\end{theorem}
This is the sharpest estimate which we could obtain in the exponential case. The proof is quite delicate, since the estimate from Theorem~\ref{th:gen_exp} is not strong enough. Theorem 1.3 extends recent results by Ascione \emph{et al.}~\cite{Ascione} who show two-sided resolvent estimates of a similar type, but for relativistic case only, i.e.\ for $\alpha > m$.

Summing up, our estimates reveal two interesting phenomena regarding the decay rates of resolvent kernels:
First, the resolvent densities $g_{\alpha}(x)$ are comparable with $\nu(x)$ at infinity for all $\alpha >0$ if, and only if, the profile of $\nu(x)$ is subexponential, see Theorem \ref{th:subexp}. This demonstrates, in particular, the transition from the subexponential to the exponential regime. 
Secondly, in the case of exponentially decaying profiles -- this was initially observed by Carmona, Masters, and Simon in the relativistic case --, the rate of decay of $g_\alpha(x)$ depends on the position of $\alpha >0$ relative to a critical parameter $\omega^*(\kappa)$, with a sharp and continuous transition in the exponent $\gamma_{\alpha}$, see Theorem \ref{th:exp} and \ref{th:exp_sharp} for the rather delicate argument. Specifically, $g_{\alpha}(x)$ is comparable with $\nu(x)$ at infinity if, and only if, $\alpha > \omega^*(\kappa)$. 
Both effects translate to the decay rates of bound states, which we analyse in detail in Section \ref{sec:applications}, see Corollaries \ref{cor:bound_state} and \ref{cor:bound_state_exp}, and the concrete Examples \ref{ex:ex1} and \ref{ex:ex2}.

The rest of this paper is organized as follows. In Section \ref{sec:conv}, we provide preliminaries on convolution semigroups and establish general estimates, including Lemma \ref{lem:exp_mom} and Theorem~\ref{th:gen_exp}, which are crucial for the subsequent analysis. Sections \ref{sec:subexp} (subexponential case) and \ref{sec:exp} (exponential case) contain the proofs of our resolvent kernel bounds. Finally, in Section \ref{sec:bound-states}, we derive estimates for bound states.

\subsection*{Notation} Two-sided estimates between functions are sometimes indicated by
\begin{gather*}
    f(x)\asymp g(x),\; x\in A
    \iff
    \exists C\geq 0 \;\forall x\in A \::\: C^{-1}f(x)\leq g(x)\leq C f(x).
\end{gather*}
The notation $f(x)\casymp{C}g(x)$ is used to highlight the comparison constant $C$. The symbol $B_r(x)$ denotes an open Euclidean ball of radius $r>0$ centred at $x \in \Rd$; $a\wedge b$ and $a\vee b$ are the minimum and maximum of $a$ and $b$, respectively. By $f_+(x)$ and $x_+$ we mean the positive parts $\max\{f(x), 0\}$ and $\max\{x,0\}$. 

\section{Applications to bound state decay and examples} \label{sec:applications}

Our results can be applied to analyse precisely  the decay of the bound states for non-local Schr\"odinger operators $H= -L+V$ acting in the complex Hilbert space $L^2(\Rd)$. We need a few preparations and further assumptions. Recall that the quadratic form $(\Ecal^{(-L)},\Dcal(\Ecal^{(-L)}))$ corresponding to $-L$ is given by
\begin{align*}
  \Ecal^{(-L)}(u,u)   &= \int_{\Rd} \Psi(\xi) |\Fourier(u)(\xi)|^2 \,d\xi, \\
  \Dcal(\Ecal^{(-L)}) &= \left\{v\in L^2(\Rd):\: \int_{\Rd} \Psi(\xi) |\Fourier(v)(\xi)|^2 \,d\xi < \infty \right\}.
\end{align*} 

We impose the following assumption on the potential $V$.
\begin{itemize}
\item[\bfseries(\namedlabel{H1}{H1})] Let 
\begin{gather}\label{eq:ass_potential_0}
	V \in L^1_{\loc}(\Rd), \quad V \leq 0,
\intertext{be such that}\label{eq:ass_potential}
	V(x) \to 0 \text{\ \ as\ \ } |x| \to \infty,
\end{gather}
and assume that $V$ is relatively form-bounded with respect to $L$ such that the relative bound is less than one, i.e.\ there are $a \in (0,1)$ and $b>0$ such that
\begin{gather} \label{eq:relative_bdd}
	|\left\langle V u, u \right\rangle|  \leq a \Ecal^{(-L)}(u,u) + b \left\|u\right\|^2_{L^2}, 
	\quad u \in \Dcal(\Ecal^{(-L)}).
\end{gather}
\end{itemize}

By the \textbf{Kato-Lions-Lax-Milgram-Nelson (KLMN) theorem} (see e.g.\ \cite[Theorem 10.21]{Schmudgen}), there exists a unique, lower bounded, self-adjoint operator $H$, called the form sum of the operators $-L$ and $V$ (we simply write $H=-L+V$), such that the form $\left(\Ecal^H, \Dcal(\Ecal^H)\right)$ of $H$ satisfies 
\begin{align} \label{eq:eq_forms}
\Dcal(\Ecal^H) = \Dcal(\Ecal^{(-L)}) \quad \text{and} \quad \Ecal^H(u,v) = \Ecal^{(-L)}(u,v) + \int_{\Rd}V(x) u(x)\overline{v(x)}\, dx, \quad u,v \in \Dcal(\Ecal^H).
\end{align}

The corresponding Schr\"odinger semigroup $\left\{e^{-tH}, t  \geq 0\right\}$ consists of bounded, self-adjoint operators such that $\left\|e^{-tH}\right\|_{L^2,L^2} \leq e^{w t}$, $t>0$, for some $w \geq 0$, see \cite[Theorem 6.2]{Casteren}. We further require the following regularity condition:
\begin{itemize}
\item[\bfseries(\namedlabel{H2}{H2})] There exists some $r >1$ such that for every $t>0$ and $\phi \in L^2(\Rd)$,
\begin{gather*}
	e^{-tH} \phi \in L^2(\Rd) 
	\text{\ has a version which is continuous and bounded on\ } \{y \in \Rd: |y|\geq r \};
\end{gather*}
\end{itemize}
The assumptions \eqref{H1}--\eqref{H2} apply to a large class of L\'evy operators $L$ and potentials $V$. 
If $V$ belongs to the Kato class corresponding to a given $L$, then these conditions typically hold -- a probabilistic approach based on the Feynman--Kac formula is in \cite{Demuth-Casteren}, an analytic approach using perturbation arguments of the kernels is given in \cite{Bogdan-Hansen-Jakubowski,Grzywny-Kaleta-Sztonyk}. The Kato class of $L$ includes bounded functions and functions with singularities, which are subcritical with respect to $L$. However, \eqref{H1}--\eqref{H2} are also satisfied for more singular models, e.g.\ for fractional Laplacians with Hardy-type potentials, and for operators which are close (in a certain sense) to fractional Laplacians with Hardy-type potentials \cite{Bogdan-Grzywny-Jakubowski-Pilarczyk, Jakubowski-Kaleta-Szczypkowski}. This includes a large class of L\'evy operators with Coulomb potentials.

We call $\phi \in L^2(\Rd)$ a \textbf{bound state} of $H$ if
\begin{align} \label{eq:eigenequation}
\text{there exists a number $\lambda<0$ such that\ \ }  H\phi = \lambda \phi.
\end{align}
By \eqref{H2}, $\phi$ has a version which is continuous on $\{x: |x| \geq r\}$. We will always work with this version. Moreover, $\phi_0 \in L^2(\Rd)$ satisfying \eqref{eq:eigenequation} with $\lambda_0:= \inf \sigma(H) <0$ is called the \textbf{ground state} of $H$; $\phi_0$ is unique and strictly positive \cite[Theorem XIII.44]{Reed-Simon}.

\begin{corollary}[L\'evy densities with subexponential profiles]\label{cor:bound_state}
Assume \eqref{L1}--\eqref{L2}, \eqref{H1}--\eqref{H2} and \eqref{eq:sub_exp}. Let $\phi$ be a bound state with eigenvalue $\lambda<0$, cf.\ \eqref{eq:eigenequation}, and let $\phi_0$ be the ground state with eigenvalue $\lambda_0 \leq \lambda  <0$. Then the following assertions hold, no matter how small $|\lambda|$ and $|\lambda_0|$ might be: 
\begin{enumerate}
\item\label{cor:bound_state-a}
There exist $C>0$ and $\rho>r$ such that
\begin{gather*}
	|\phi(x)| \leq C f(|x|), \quad |x| \geq \rho,
\end{gather*}

\item\label{cor:bound_state-b} 
There exist $C>1$ and $\rho>r$ such that
\begin{gather*}
	C^{-1} f(|x|) \leq \phi_0(x) \leq C f(|x|), \quad |x| \geq \rho.
\end{gather*}
\end{enumerate}
\end{corollary} 

Next, we give estimates for the exponential case. Recall that
\begin{gather*}
\gamma_{\alpha}(\theta) = 
\begin{cases}
	\kappa,\ & \alpha > \omega(\kappa \theta),\\
	\omega(\tinybullet\theta)^{-1}(\alpha),\ & 0 < \alpha \leq \omega(\kappa \theta),
\end{cases}
\end{gather*} 
and $\omega^*(\kappa):= \sup_{\theta \in \sphere^{d-1}} \omega(\kappa \theta)$.

The following result is formulated for \textbf{radial} L\'evy measures only as we want to make it as sharp as possible.
In this case the function $\omega$ is radial, and the map $(0,\infty) \ni \alpha \mapsto \gamma_{\alpha}$ is a continuous function which does not depend on $\theta \in \sphere^{d-1}$ (so we omit $\theta$ in the notation). Moreover, if $\alpha < \omega^*(\kappa)$, then $\gamma_{\alpha} < \kappa$, and $\gamma_{\alpha} \downarrow 0$ as $\alpha \downarrow 0$. 

\begin{figure}[h!]\centering
	\mbox{}\hfill
	\includegraphics[width=0.4\linewidth]{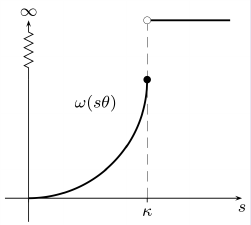}\hfill
	\includegraphics[width=0.5\linewidth]{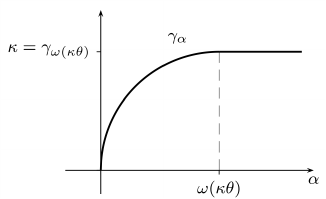}\hfill\mbox{}
	\caption{The left panel shows the function $s\mapsto \omega(s\theta)$ for any fixed $\theta\in\sphere^{d-1}$ in the relativistic case, i.e.\ $\Psi(\xi) = \sqrt{|\xi|^2+m^2}-m$ (note that we are in the radial case). It is strictly convex and continuous in $[0,\kappa]$, with finite and positive derivatives of any order in the interval $[0,\kappa)$. In the open interval $(\kappa,\infty)$ it is infinite. The (generalized) inverse $\alpha\mapsto\gamma_\alpha$ is shown in the right panel. Notice that $\gamma_\alpha$ is continuous in $[0,\omega(\kappa\theta)]$, and it is continuously extended by $\kappa = \gamma_{\omega(\kappa\theta)}$ on the interval $[\omega(\kappa\theta),\infty)$. Recall that the Schr\"odinger operator $H=-L+V$ serves as the energy operator (Hamiltonian) in the mathematical model describing the motion of a (relativistic) quantum particle. The kinetic term $-L = \Psi(\frac 1i \nabla)$ is a positive, homogeneous pseudo-differential operator. Here, $\frac 1i \nabla$ is the momentum operator in a multidimensional setting, consistent with the quantization rules of quantum mechanics. The threshold $\omega(\kappa\theta) = - \Psi(\frac 1i \kappa\theta)$ seems to relate to the particle's kinetic energy. It is critical in the sense that $\omega(\xi) = \infty$ whenever $|\xi| > \kappa$.
}
\label{pic-graph-gamma}
\end{figure} 

\begin{corollary}[L\'evy densities with exponential profiles]\label{cor:bound_state_exp}
	Assume \eqref{L1}--\eqref{L2} and \eqref{H1}--\eqref{H2}, and let $\nu$ be radial, i.e.\ $\nu(x)=\nu(|x|)$, and such that \eqref{eq:exp_prof}--\eqref{eq:g_sub_exp} hold. Let $\phi$ be a bound state with eigenvalue $\lambda<0$, cf.\ \eqref{eq:eigenequation}, and $\phi_0$ be the ground state with eigenvalue $\lambda_0 \leq \lambda <0$.
\begin{enumerate}
\item\label{cor:bound_state_exp-a}
	If $|\lambda| > \omega^*(\kappa)$, then there exist $C>0$ and $\rho>r$ such that
	\begin{gather*}
	|\phi(x)| \leq C \exp\left(-\kappa |x| \right) h(|x|), \quad |x| \geq \rho.
	\end{gather*}
\item\label{cor:bound_state_exp-b}
	If $|\lambda_0| > \omega^*(\kappa)$, then there exist $C>1$ and $\rho>r$ such that
	\begin{gather*}
	C^{-1} \exp\left(-\kappa |x| \right) h(|x|) \leq \phi_0(x) \leq C \exp\left(-\kappa |x| \right) h(|x|), \quad |x| \geq \rho.
	\end{gather*}
\item\label{cor:bound_state_exp-c} 
	If $|\lambda| \leq \omega^*(\kappa)$, then for every $\epsilon >0$ there exist $C=C(\epsilon)>0$ and $\rho>r$ such that
	\begin{gather*}
	|\phi(x)| \leq C \exp\left(-(\gamma_{|\lambda|}-\epsilon)|x|\right), \quad |x| \geq \rho.
	\end{gather*}
\item\label{cor:bound_state_exp-d} 
	If $|\lambda_0| \leq \omega^*(\kappa)$, then for every $\epsilon >0$ there exist $C=C(\epsilon)>1$ and $\rho>r$ such that
	\begin{gather*}
	C^{-1} \exp\left(-(\gamma_{|\lambda_0|}+\epsilon)|x|\right) \leq \phi_0(x) \leq C \exp\left(-(\gamma_{|\lambda_0|}-\epsilon)|x|\right), \quad |x| \geq \rho.
\end{gather*}
\end{enumerate}
\end{corollary} 
Our results generalize some estimates by Carmona, Masters and Simon \cite[Propositions IV.1 and IV.3]{Carmona-Masters-Simon} for fractional and relativistic Schr\"odinger operators. Corollary \ref{cor:bound_state_exp} also improves the upper bound in \cite[Proposition IV.2]{Carmona-Masters-Simon} for exponentially localized L\'evy measures with profiles satisfying \eqref{eq:exp_prof}--\eqref{eq:g_sub_exp}. The second assertion in Corollary \ref{cor:bound_state_exp} gives sharp two-sided bounds for the ground state, which seems to be novel even in the relativistic case. Moreover, our results allow for a better understanding of the threshold $m$ for $|\lambda|$ and of the decay rate $\sqrt{2m \alpha - \alpha^2}$ appearing in the relativistic case; see Figure~\ref{pic-graph-gamma} and Example \ref{ex:ex2} for a further discussion.

The decay of the bound states for models with L\'evy operators whose L\'evy densities have profiles satisfying \eqref{L2} has already been studied by Kaleta and L\H orinczi \cite[Theorem 4.3 and 4.1]{Kaleta-Lorinczi}. In that paper upper estimates for the decay are established with the help of probabilistic potential theory and properties of harmonic functions. This requires several technical assumptions, which are not always easy to check. Theorem 4.2 in \cite{Kaleta-Lorinczi} essentially contains a rough approximation $\eta_0$ for the energy threshold $\omega^*(\kappa)$. We actually identify identify $\omega^*(\kappa)$ in Corollary \ref{cor:bound_state_exp}. In some concrete situations, see Example~\ref{ex:ex2}, there is even a closed expression for it. The paper by Kondratiev \emph{et al.}~\cite{Kondratiev} contains estimates for the resolvent and the ground state. It can be seen as an elementary version of \cite{Kaleta-Lorinczi} since it uses only finite L\'evy measures and, therefore, allows for direct computations. Let us point out that the methods in the present paper are very different from those used in \cite{Carmona-Masters-Simon,Kaleta-Lorinczi,Kondratiev} since we use a (non-probabilistic) semigroup approach. 

We will now discuss two examples illustrating Theorem \ref{th:subexp} and \ref{th:exp_sharp} and Corollaries \ref{cor:bound_state} and \ref{cor:bound_state_exp}.

\begin{example}[transition from the subexponential to the exponential regime]\label{ex:ex1}
Let the profile of the radial L\'evy density $\nu$ be given by
\begin{gather}\label{eq:prof_ex}
	f(r) = \left(\I_{[0,1]}(r) r^{-d-\beta} + \I_{(1,\infty)}(r) r^{-\delta} \right)\exp\left(-\kappa r^{\eta}\right) ,
\end{gather}
where
\begin{gather*}
	\beta \in (0,2), \quad \kappa >0, \quad \eta \in (0,1], \quad \text{and} \quad \delta \geq 0.
\end{gather*}
We have
\begin{gather}\label{eq:psi_ex}
	\Psi(\xi) \asymp \Psi^*(|\xi|) \asymp |\xi|^{\beta} \wedge |\xi|^2, \quad \xi \in \Rd,
\end{gather}
for the full range of parameters $\beta, \kappa, \eta$ and $\delta$; in particular \eqref{L1} holds. The second assumption \eqref{L2} is satisfied if, and only if, $\eta = 1$ and $\delta > \frac 12(d+1)$, or $\eta \in (0,1)$ and  $\delta \geq 0$, see e.g.\ \cite[Lemma 8]{KaletaSztonyk2019}.

Theorem \ref{th:subexp} identifies the first sharp transition regarding the rate of decay of the resolvent kernel $g_{\alpha}(x)$ as $|x| \to \infty$ as it says that this decay is controlled by $f(|x|)$ for every $\alpha>0$ (no matter how small $\alpha>0$ may be) if, and only if,  $\eta \in (0,1)$. This property breaks down as soon as $\eta$ takes the value $1$. This transition translates to the decay rates of bound states for the corresponding Schr\"odinger operators, see Corollary \ref{cor:bound_state}.

On the other hand, if $\eta = 1$, then the decay rate of $g_{\alpha}(x)$ depends on the position of $\alpha$ relative to $\omega^*(\kappa)$. This is illustrated in the next example for a class of relativistic L\'evy operators and semigroups.  
\end{example}

\begin{example}[relativistic $\beta$-stable L\'evy operators, resp., semigroups]\label{ex:ex2}
Let $\beta \in (0,2)$ and $m>0$ and
\begin{align*}
	\nu(x) = \nu(|x|) 
	&= \frac{ \beta}{2 (4\pi)^{d/2}\Gamma\left(1-\frac 12\beta\right)} \int_0^{\infty} \exp\left(-\frac{|x|^2}{4u} - m^{\frac 2\beta} u\right) u^{-1-\frac{d+\beta}{2}} \,du \\
    &= \frac{\beta 2^{\frac{\beta-d}{2}} m^{\frac{d+\beta}{2\beta}}}{\pi^{\frac{d}{2}}\Gamma\left(1-\frac 12\beta\right)} \frac{K_{\frac{d+\beta}{2}}\left(m^{\frac{1}{\beta}}|x|\right)}{|x|^{\frac{d+\beta}{2}}}, 
    \qquad x \in \Rd \setminus \left\{0\right\},
\end{align*} 
where
\begin{gather*}
	K_{\mu}(r) 
	= \frac{1}{2} \left(\frac{r}{2}\right)^{\mu} \int_0^{\infty} u^{-\mu-1} \exp\left(-u-\frac{r^2}{4u}\right) du, 
	\quad \mu>0,\; r>0,
\end{gather*} 
is the modified Bessel function of the second kind, see e.g.\ \cite[10.32.10]{NIST}. Using the asymptotics (see \cite[10.25.3 and 10.30.2]{NIST})
\begin{gather*}
	\lim_{r \to \infty} K_{\mu}(r) \sqrt{r} e^r  = \sqrt{\pi/2}, 
	\qquad \lim_{r \to 0} K_{\mu}(r) r^{\mu}  = 2^{\mu-1} \Gamma(\mu),
\end{gather*}
we can show that
\begin{gather*}
	\nu(|x|) \asymp f(|x|), 
	\text{\ \ where\ \ } 
	f(r) = \left(\I_{[0,1]}(r) r^{-d-\beta} + \I_{(1,\infty)}(r) r^{-(d+\beta+1)/2} \right)e^{-m^{1/\beta}r },
\end{gather*}
i.e.\ \eqref{eq:exp_prof}--\eqref{eq:g_sub_exp} hold with 
\begin{gather*}
	\kappa = m^{1/\beta} 
	\text{\ \ and\ \ }
	h(r) = \I_{[0,1]}(r) r^{-d-\beta} + \I_{(1,\infty)}(r) r^{-(d+\beta+1)/2}.
\end{gather*}
Observe that this is a special form of an exponential L\'evy density, resp., profile discussed in Example \ref{ex:ex1} above, cf.\ \eqref{eq:prof_ex}. In particular, \eqref{L1}--\eqref{L2} hold. We have 
\begin{gather*}
	\Psi(\xi) = \left(|\xi|^2+m^{2/\beta}\right)^{\beta/2} - m,
\end{gather*}
and the L\'evy operator associated with $\nu$ is the \textbf{relativistic ($\beta$-stable) operator}
\begin{gather*}
	L = -\left(-\Delta+m^{2/\beta}\right)^{\beta/2}+ m. 
\end{gather*}
The threshold, which determines the transition in the exponential decay rates in Theorems \ref{th:exp}, \ref{th:exp_sharp} and Corollary \ref{cor:bound_state_exp}, is given by 
\begin{gather*}
	\omega^*\left(m^{1/\beta}\right) = m, 
	\text{\ \ where\ \ }
	\omega(\xi) = m - \left(m^{2/\beta}-|\xi|^2\right)^{\beta/2}, 
	\quad |\xi| \leq m^{1/\beta}. 
\end{gather*} 
The rate $\gamma_{\alpha}$ appearing in Theorems \ref{th:exp} and \ref{th:exp_sharp} can be easily computed by inverting $\omega(\cdot)$, which gives
\begin{gather*}
	\gamma_{\alpha} = \sqrt{m^{2/\beta}-(m-\alpha)_{+}^{2/\beta}}, \quad \alpha >0.
\end{gather*}
In particular, we have $\gamma_{|\lambda|} = \sqrt{m^{2/\beta}-(m-|\lambda|)_{+}^{2/\beta}}$ in Corollary \ref{cor:bound_state_exp}, and for $\beta=1$ we recover the result of Carmona, Masters and Simon \cite[Proposition IV.1 and IV.3]{Carmona-Masters-Simon} quoted above. 
\end{example}
 
\section{Convolution semigroups} \label{sec:conv}

Let $\{ \mu_t,\, t\geq 0 \}$ be a convolution semigroup of probability measures such that the Fourier transform is of the form $\Fourier(\mu_t)(\xi)=\int_{\Rd} e^{i\scalp{\xi}{y}}\,\mu_t(dy)=\exp(-t\Psi(\xi))$ with a characteristic exponent of the form
\begin{gather*}
  \Psi(\xi) =    \int \left(1-\cos(\scalp{\xi}{y}) \right)\nu(y)\, dy ,\quad \xi\in\Rd.
\end{gather*}
By $\Psi^*$ we define the maximal function, and $\Psi^*_{-}$ is the generalized inverse, see \eqref{eq:maximal} and \eqref{eq:maximal_inv} in Section~\ref{intro}.

From \cite[Lemma 5(a)]{KaletaSztonyk2017} we know that there exists a constant $C_4 \in (0,1]$ such that 
\begin{gather*}
	C_4 \Psi^*(|x|) \leq \Psi(x) \leq \Psi^*(|x|), \quad x \in \Rd.
\end{gather*}
We also note that the functions $t\mapsto\Psi^*_{-}(t)$ and $t\mapsto 1\big/\Psi^*_{-}\left(\frac{1}{t}\right)$ are both increasing. Recall that \eqref{L1} is equivalent to the lower scaling property \eqref{eq:scaling}. By \cite[Proposition 3.6]{GrzywnySzczypkowski2020} 
this is also equivalent to the existence of $\widetilde C_1>0$ such that
\begin{gather*}
	\int_{\Rd} e^{-t\Psi(\xi)}|\xi|\, d\xi 
	\leq \widetilde C_1 \left(\Psi^*_{-}\left( \tfrac 1t\right)\right)^{d+1},
  	\quad t>0.
\end{gather*}

\begin{lemma}\label{h_upper_est} 
	If \eqref{L1} holds, then there exists a constant $C_5 >0$ such that
	\begin{equation}\label{h_upper_est_eq}
		\frac 1{\Psi^*_{-}\left(\frac{1}{t}\right)} 
		\leq C_5 t^{1/\alpha}, \quad t\in [C_5^{-\alpha},\infty),
	\end{equation}
	with $\alpha$ coming from \eqref{eq:scaling}. 
\end{lemma}
\begin{proof}
From \eqref{eq:scaling} we get
\begin{gather*}
		\Psi^*(1) = \Psi^*\left(\tfrac{1}{r}r\right) \geq C_2 r^{-\alpha}\Psi^*(r), \quad r\in (0,1],
\end{gather*}
	hence $\Psi^*(r) \leq c_1 r^{\alpha}$,  for $r\in (0,1]$,  where $c_1=C_2^{-1}\Psi^*(1)$.  This yields
\begin{gather*}
r\leq \Psi^*_{-}(\Psi^*(r)) \leq \Psi^*_{-}(c_1 r^{\alpha}), \quad r\in (0,1],
\end{gather*}
or, equivalently,
	\begin{gather*}
		\Psi^*_{-}\left(\tfrac 1t\right) \geq (c_1 t)^{-1/\alpha}, \quad t\geq \frac{1}{c_1},
	\end{gather*}
	which proves the claim.
\end{proof}

We will also need the following property of the profile $f$, which holds under \eqref{L2}: for every $r >0$ there exists a constant $C_6=C_6(r) \geq 1$ such that
\begin{align} \label{eq:comp}
	f(s-r) \leq C_6 f(s), \quad s \geq 3r,
\end{align}
see e.g.\ \cite[Lemma 1 b)]{KaletaSztonyk2019}. For instance, it leads to the following lemma.

\begin{lemma}\label{pxy}
If \eqref{L1} and \eqref{L2} hold, then there exists 
$C_7 \in (0,1)$ such that
\begin{gather*}
	p_t(x-y) \geq C_7 p_t(x), \quad  |x|\geq 1, \; |y|\leq \frac 12, \; t>0.
\end{gather*}
\end{lemma}
\begin{proof}
Case 1: $t\in [0,1]$: By \eqref{eq:small_time}, \eqref{eq:comp} and the monotonicity of $f$ there are constants $c_1, c_2, c_3>0$ such that
\begin{gather*}
	p_t(x-y) \geq c_1 t f(|x-y|) \geq c_2 t f(|x|) \geq c_3 p_t(x), \quad  |x|\geq 1,\; |y|\leq \frac 12,\; t\in (0,1].
\end{gather*} 

\smallskip\noindent
Case 2: $t>1$. We first extend the estimate obtained in Case 1 to all $x \in \Rd$ when $t=1$.
By \eqref{L1}, 
\begin{gather*}
	p_1(x) = (2\pi)^{-d} \int e^{-\Psi(u)}e^{i\scalp{x}{u}}\, du \leq (2\pi)^{-d} \int e^{-\Psi(u)}\, du  <\infty, \quad x\in\Rd.
\end{gather*}  
On the other hand, from \cite[Theorem 2]{KaletaSztonyk2015} we know that $p_1(x) \geq c_4>0$ for $|x|\leq 3/2$.  Hence, we have $p_1(x-y)=p_1(y-x) \geq c_5 p_1(x)$ for all $x\in\Rd$ and $|y|\leq 1/2$. Using this bound and the Chapman--Kolmogorov equations, we get for $x \in \Rd$, $|y|\leq 1/2$ and $t>1$
\begin{gather*}
	p_t(y-x) 
	= \int p_{t-1}(z-x)p_1(y-z)\, dz 
	\geq c_5 \int p_{t-1}(z-x)p_1(z)\, dz 
	= c_5 p_t(x),
\end{gather*}
which completes the proof.
\end{proof}

We will now discuss the estimates for densities corresponding to restricted L\'evy measures. We write
\begin{gather*}
	\nulj_r(x)
	:=\nu(x) \I_{(r,\infty)}(|x|), 
	\quad\text{and}\quad 
	\nusj_r(x) := \nu(x) \I_{(0,r]}(|x|), \qquad r>0,
\end{gather*}
and denote the corresponding densities by $\plj_t^r$ and $\psj_t^r$. The density $\plj_t^r$ is of compound Poisson type, i.e.
\begin{gather} \label{eq:def_Poiss}
	\plj_t^r(x)=e^{-t|\nulj_r|}\sum_{n=1}^\infty \frac{t^n\nulj_r^{n*}(x)}{n!},
	\quad
	 |\nulj_r| := \nu\left(\{|y|>r\}\right),
\end{gather} 
and the Fourier transform (characteristic function) of the density  $\psj_t^r$ is given by
\begin{gather*}
	\Fourier{\psj_t^r}(u) = \exp\left(-t\int \left(1-\cos(\scalp{u}{y})\right) \nusj_r(dy)\right).
\end{gather*}
Our approach is based on the following decomposition 
\begin{gather} \label{eq:decomp}
	p_t(x) = e^{-t|\nulj_r|} \psj_t^r(x) + \psj_t^r\ast \plj_t^r, \quad r>0, \;t>0 \,.
\end{gather}
The following estimate, taken from \cite[Theorem 6]{KnopovaSchilling2012}, is fundamental for our investigations:
\begin{gather} \label{eq:KnopovaSchilling}
	\psj_t^r(x) \leq \psj_t^r(0) e^{-\scalp{\xi}{x}+t\omega_r(\xi)}, \quad x,\xi\in\Rd,\; t>0,\; r>0,
\end{gather}
where 
\begin{gather*}
	\omega_r(\xi) 
	= \int \left(\cosh (\scalp{\xi}{y}) - 1\right) \nusj_r(dy) 
	= \int_{|y| \leq r} \left(\cosh (\scalp{\xi}{y}) - 1\right) \nu(dy).
\end{gather*}
Observe that we have for $t \geq 1$ and $r \geq 1$
\begin{align*}
	\psj_t^r(0) 
	& = (2\pi)^{-d} \int \exp\left(-t\int \left(1-\cos(\scalp{u}{y})\right) \nusj_r(dy)\right) du \\
	& \leq (2\pi)^{-d} \int e^{-\Psi(u)}\exp\left( \int_{|y|>r} \left(1-\cos(\scalp{u}{y})\right) \nu(dy)\right) du \\ 
	& \leq (2\pi)^{-d} e^{2\nu(B(0,r)^c)} \int e^{-\Psi(u)}\, du \\
	& \leq (2\pi)^{-d} e^{2\nu(B(0,1)^c)} \int e^{-\Psi(u)}\, du < \infty,
\end{align*}
which implies that 
\begin{align} \label{eq:small_est}
	\psj_t^r(x) \leq C_8 e^{-\scalp{\xi}{x}+t\omega_r(\xi)}, \quad x,\xi\in\Rd, \;t, r \geq 1,
\end{align}
with a uniform constant $C_8$.

We can now estimate the density $\plj_t^r$. In order to keep notation simple, we define
\begin{gather*}
 	K(r):=\sup_{|x|\geq 1} \frac{\displaystyle\int_{|y-x| >r,\, |y|>r} \nu(x-y)\nu(y)\, dy}{\nu(x)}, \quad r \geq 1.
\end{gather*}
Clearly, $C_0^{-3} K_f(r) \leq K(r) \leq C_0^3 K_f(r)$, $r \geq 1$, where $K_f$ is from \eqref{L2}.

\begin{lemma}\label{convolutions} 
	Assume \eqref{L2}. For every $r \geq 1$ there exists a constant $C_9 = C_9(r) \geq 1$ such that
	\begin{gather*}
		\nulj_{r}^{*n}(x) \leq C_9 n (K(r)\vee |\nulj_r|)^{n-1}\nu(x),\quad |x|\geq 1,\; n\in\nat,
	\intertext{and}
  		\plj_t^r(x) \leq C_9 t e^{-t|\nulj_r|}e^{(K(r)\vee |\nulj_r|)t}\nu(x),
		\quad |x|\geq 1,\; t>0.
	\end{gather*}
\end{lemma}
\begin{proof}
The second estimate follows directly from the first estimate because of \eqref{eq:def_Poiss}. Let us consider the $n$-fold convolution. We proceed by induction.

For $n=1$ the assertion is trivial, and for $n=2$ it follows directly  from the definition of the function $K$. We will verify the induction step $n\rightsquigarrow n+1$. For $|x|\geq 1$ we have
\begin{gather*}
	\nulj_r^{*(n+1)}(x) 
	= \int_{|y|> r} \nulj_r(x-y)\nulj_r^{*n}(y)\, dy + \int_{|y| \leq r} \nulj_r(x-y)\,\nulj_r^{*n}(y)\, dy 
	=: \mathrm{I} + \mathrm{II},
\end{gather*}
and, using the induction hypothesis and the definition of the function $K$, we get
\begin{gather*}
	\mathrm{I} \leq C_9 n (K(r)\vee |\nulj_r|)^{n-1} \int_{\substack{|x-y|> r\\ |y|> r}} \nu(x-y) \nu(y)\, dy
	\leq C_9 n (K(r)\vee |\nulj_r|)^{n} \nu(x).
\end{gather*}
By \eqref{eq:comp}, for $|x|\geq 3r$ and $|y| \leq r$, we get $\nulj_r(x-y) \leq C_0 f(|x-y|) \leq C_0 C_6 f(|x|) \leq C_0^2 C_6 \nu(x)$,
which yields
\begin{gather*}
	\mathrm{II} 
	\leq C_0^2 C_6 |\nulj_r^{*n}| \nu(x)
	= C_0^2 C_6 |\nulj_r|^n \nu(x).
\end{gather*}
If $|x| < 3r$ and $|x-y|> r$, we have $\nulj_r(x-y) \leq C_0 f(|x-y|) \leq C_0 f(r)$. Hence,
\begin{gather*}
	\mathrm{II} 
	\leq C_0 (f(r)/f(3r)) f(3r) |\nulj_r^{*n}| 
	\leq C_0 (f(r)/f(3r)) f(|x|) |\nulj_r|^n 
	\leq C_0^2 (f(r)/f(3r)) |\nulj_r|^n \nu(x). 
\end{gather*}
We then see that the lemma follows with $C_9=C_0^2 (C_6 \vee (f(r)/f(3r)))$.
\end{proof}

Next, we show the finiteness of exponential moments (cf.\ \cite[Lemma 2]{KaletaSztonyk2019}) for L\'evy measures with a profile $f$ given by \eqref{eq:exp_prof}--\eqref{eq:g_sub_exp}.
\begin{lemma}\label{lem:exp_mom}
	Assume that the profile $f$ is given by \eqref{eq:exp_prof}--\eqref{eq:g_sub_exp}. Under \eqref{L2}, we have for every $\xi \in \Rd$ such that $|\xi| \leq \kappa$
\begin{gather*}
	\int_{|y| \geq 1} e^{\scalp{\xi}{y}}\nu(y)\,dy < \infty.
\end{gather*}
\end{lemma}
\begin{proof}
By \eqref{L2}, we have
\begin{gather}\label{eq:aux_int}
	\int_{\substack{|y| \geq 1\\ |x-y| \geq 1}} \exp\left(\kappa(|x|-|x-y|)\right) \exp\left(\log h(|x-y|)-\log h(|x|)\right) f(|y|) \,dy 
	\leq K_f(1) < \infty,
\end{gather}
whenever $|x| > 2$. Since the map $r \mapsto \log h(r)/r$ is eventually monotone, we have for any fixed $y$ such that $|y| \geq 1$ and sufficiently large $|x|$
\begin{align*}
	&\exp\left(\log h(|x-y|)-\log h(|x|)\right) \\
   	&\quad = \left(\I_{\left\{|x| < |x-y|\right\}} + \I_{\left\{|x| \geq  |x-y|\right\}} \right) 
   		\exp\left[\log h(|x-y|)-\log h(|x|)\right] \\
  	&\quad \geq \I_{\left\{|x| < |x-y|\right\}} \cdot \exp\left[|x-y| \frac{\log h(|x-y|)}{|x-y|}-|x|\frac{\log h(|x|)}{|x|}\right] 
	+ \I_{\left\{|x| \geq  |x-y|\right\}}  \\
  	&\quad =\I_{\left\{|x| < |x-y|\right\}} \cdot \exp\left[\left(|x-y| -|x|\right)\frac{\log h(|x|)}{|x|}\right]
	\exp\left[|x-y| \left(\frac{\log h(|x-y|)}{|x-y|}-\frac{\log h(|x|)}{|x|}\right)\right] \\
	& \phantom{\exp\left[\left(|x-y| -|x|\right)\frac{\log h(|x|)}{|x|}\right]
		\exp\left[|x-y| \left(\frac{\log h(|x-y|)}{|x-y|}-\frac{\log h(|x|)}{|x|}\right)\right]}  
	    \qquad\mbox{} + \I_{\left\{|x| \geq  |x-y|\right\}}  \\
	&\quad \geq \I_{\left\{|x| < |x-y|\right\}} \cdot \exp\left[\left(|x-y| -|x|\right)\frac{\log h(|x|)}{|x|}\right] + \I_{\left\{|x| \geq  |x-y|\right\}} .
\end{align*}
Suppose first that $|\xi| = \kappa$, and consider $x = (r/\kappa) \xi$.  Letting $r\to\infty$, we see that $\log h(|x|)/|x| \to 0$, hence we obtain 
\begin{gather*}
	\liminf_{r \to \infty} \exp \left(\log h(|x-y|)-\log h(|x|)\right) \geq 1, \quad |y| \geq 1.
\end{gather*}
Moreover, $|x-y| -|x| \to - \scalp{\xi}{y}/\kappa$ as $r\to\infty$ and Fatou's lemma, applied to the integral in \eqref{eq:aux_int}, 
shows
\begin{gather*}
	\int_{|y| \geq 1} e^{\scalp{\xi}{y}} f(|y|)\, dy \leq K_f(1) < \infty.
\end{gather*}
Suppose now that $0<|\xi| < \kappa$. Substituting $y = (|\xi|/\kappa) z$ we arrive at 
\begin{gather*}
	\infty 
	> \int_{|y| \geq \frac{|\xi|}{\kappa}} e^{\scalp{ \frac\kappa{|\xi|} \xi}{y}} f(|y|) \,dy 
    = \left(\frac{|\xi|}{\kappa}\right)^d \int_{|z| \geq 1} e^{\scalp{\xi}{z}} f((|\xi|/\kappa)|z|) \,dz
	\geq \left(\frac{|\xi|}{\kappa}\right)^d \int_{|z| \geq 1} e^{\scalp{\xi}{z}} f(|z|)\, dz.
\end{gather*}
As the case $\xi = 0$ is trivial, this completes the proof.
\end{proof}

We close this section by proving an upper bound of the kernel $p_t(x)$ for a general convolution semigroup, which corresponds to a L\'evy measure having a certain finite exponential moment.
\begin{theorem} \label{th:gen_exp}
	Let $\nu$ be a symmetric L\'evy measure on $\Rd \setminus\left\{0\right\}$ and $\{ \mu_t:\, t\geq 0 \}$ a convolution semigroup of probability measures such that $\Fourier(\mu_t)(\xi)=\int_{\Rd} e^{i\scalp{\xi}{y}}\,\mu_t(dy)=\exp(-t\Psi(\xi))$. The characteristic exponent is assumed to be of the form 
	\begin{gather*}
		\Psi(\xi) =    \int \left(1-\cos(\scalp{\xi}{y}) \right) \nu(dy),\quad \xi\in\Rd,
	\end{gather*}
	satisfying
	\begin{gather} \label{eq:integr}
		\int e^{-t\Psi(z)} \,dz < \infty, \quad t>0.
	\end{gather}
	Let $\xi_0\in\Rd$ be arbitrary. If 
	\begin{gather*}
		\int_{|y| \geq 1} e^{\scalp{\xi_0}{y}}\nu(y)\,dy < \infty,
	\intertext{then}
		p_t(x) \leq p_t(0) e^{-\scalp{\xi_0}{x} + t\omega(\xi_0)}, \quad x \in \Rd,\; t>0.
	\end{gather*}
\end{theorem} 

\begin{remark}
	The condition \eqref{eq:integr} is assumed just for convenience. It can be replaced by any condition that ensures the existence of all densities appearing in the proof (also for small jumps). Knopova and Schilling, see \cite[Theorem~6]{KnopovaSchilling2012}, \textbf{assume} the existence of densities without giving conditions. 
\end{remark}

\begin{proof}[Proof of Theorem \ref{th:gen_exp}]
By \eqref{eq:decomp} and \eqref{eq:KnopovaSchilling}, we have for every $x \in \Rd$, $t >0$ and $r>1$
\begin{align*}
	p_t(x) 
	&= \int_{\Rd} \psj_t^r (x-y) \left(e^{-t|\nulj_r|} \,\delta_0(dy) +  \plj_t^r(y)dy\right)  \\
	&\leq  \psj_t^r(0) e^{-\scalp{\xi_0}{x} +t\omega_r(\xi_0)} \int_{\Rd} e^{\scalp{\xi_0}{y}} \left(e^{-t|\nulj_r|} \,\delta_0(dy) 
	+  \plj_t^r(y) dy\right) .
\end{align*}
Since 
\begin{gather*}
	\int_{|y| \geq 1} e^{\scalp{\xi_0}{y}}\ \nulj_r(y)\,dy < \infty, \quad r \geq 1,
\intertext{the function}
	\bar{\omega}_r(\xi_0) 
	= \int \left(\cosh (\scalp{\xi_0}{y}) - 1\right) \nulj_r(dy) 
	= \int_{|y|>r} \left(\cosh (\scalp{\xi_0}{y}) - 1\right) \nu(dy)
\end{gather*}
is well-defined. By \cite[Theorem 25.17]{Sato},
\begin{gather*}
	\int_{\Rd} e^{\scalp{\xi_0}{y}} \left(e^{-t|\nulj_r|} \,\delta_0(dy) +  \plj_t^r(y) dy\right) = e^{t \bar{\omega}_r(\xi_0)},
\end{gather*}
and consequently,
\begin{gather*}
	p_t(x) \leq  \psj_t^r(0) e^{-\scalp{\xi_0}{x} +t\omega(\xi_0)}, \quad x \in \Rd,\; t>0,\; r >1.
\end{gather*}
A further application of \eqref{eq:decomp} yields $\psj_t^r(0)	\leq e^{t|\nulj_r|} p_t(0)$, $t>0$, $r >1$. Letting $r \to \infty$ in the estimate
\begin{gather*}
	p_t(x) \leq  e^{t|\nulj_r|} p_t(0) e^{-\scalp{\xi_0}{x} +t\omega(\xi)}, \quad x \in \Rd,\; t>0,
\end{gather*}
finishes the proof.
\end{proof}

\section{Subexponential decay}\label{sec:subexp}

In this section we prove Theorem \ref{th:subexp} in a series of lemmas. Our argument will be as follows:
\begin{gather*}
	\ref{th:subexp-A}
	\xLeftrightarrow{\text{L.\ref{l-41}}}\ref{th:subexp-B}
	\xLeftrightarrow[\text{L.\ref{l-43}}]{\text{L.\ref{l-42}}}\ref{th:subexp-C}
	\xLongrightarrow{\text{L.\ref{l-44}}}\ref{th:subexp-D}
	\xLongrightarrow{\text{L.\ref{l-44}}}\ref{th:subexp-E}
	\xLongrightarrow{\text{L.\ref{l-45}}}\ref{th:subexp-C}
\end{gather*}

\begin{lemma}\label{l-41}
	The assertions \ref{th:subexp}.\ref{th:subexp-A} and \ref{th:subexp}.\ref{th:subexp-B} are equivalent.
\end{lemma}
\begin{proof} 
	First we will prove \ref{th:subexp-A}$\Rightarrow$\ref{th:subexp-B}. If \ref{th:subexp-A} holds, then for every $\epsilon>0$ there exists $r_0>0$ such that for every $r\geq r_0$
	\begin{gather*}
		\frac{\log f(r)}{r} \geq -\epsilon,
	\intertext{hence}
		f(r) \geq e^{-\epsilon r}, \quad r\geq r_0.
	\end{gather*}
	For $r\in (1,r_0]$ we obtain from the monotonicity of $f$ that $f(r) \geq f(r_0) \geq e^{-\epsilon r_0} \geq e^{-\epsilon r_0} e^{-\epsilon r}$,  hence \ref{th:subexp-B} follows with $\tilde{C}=e^{-\epsilon r_0}$. 

  	Conversely, we assume \ref{th:subexp-B}. Fix $\epsilon >0$. Using the monotonicity of $f$ and the inequality \ref{th:subexp-B} with $\epsilon/2$ we get 
	\begin{gather*}
		\frac{1}{r}\log f(1) \geq \frac{1}{r}\log f(r) \geq  \frac{\log \tilde{C}}{r} - \frac{\epsilon}{2},\quad r\geq 1.
	\end{gather*}
	For $r\geq \frac{1}{\epsilon}\max\left\{\log f(1), -2\log\tilde{C},\epsilon\right\}$ we have
	\begin{gather*}
		 \epsilon \geq \frac{\log f(r)}{r} \geq -\epsilon,
	\end{gather*}
	hence, $\lim_{r\to \infty} \frac{\log f(r)}{r}=0$.  
\end{proof}

\begin{lemma}\label{l-42}
	If \eqref{L1} and \eqref{L2} hold, then \ref{th:subexp}.\ref{th:subexp-B} implies \ref{th:subexp}.\ref{th:subexp-C}.
\end{lemma}
\begin{proof} 
Fix $\alpha_0>0$. If $t \in (0,1]$, then we get from \eqref{eq:small_time}
\begin{gather*}
	p_t(x) \leq c_1 t f(|x|) \leq c_1 e^{\alpha_0 t} f(|x|), \quad |x| \geq 1.
\end{gather*}
For $t>1$ and $1 \leq |x|\leq 3$ the estimate in \ref{th:subexp-C} follows from the boundedness of $p_t$ and the monotonicity
of $f$. Indeed, 
\begin{align*}
	p_t(x) 
	&= (2\pi)^{-d} \int e^{-t\Psi(u)}e^{i\scalp{x}{u}}\, du 
	\leq (2\pi)^{-d} \int e^{-\Psi(u)}\, du\\
	&= c_2 \frac{f(|x|)}{f(|x|)} 
	\leq \frac{c_2}{f(|3|)} f(|x|) \leq c_3 e^{\alpha_0 t} f(|x|).
\end{align*}
All that remains is to consider the case $t> 1$ and $|x|>3$.  We will use the decomposition \eqref{eq:decomp}.

First of all, by using \eqref{L2}, we can find $r \geq 1$ such that $K(r) \leq \frac 14\alpha_0$ and, consequently,
\begin{align}\label{eq:choose_r}
	te^{K(r)t} \leq c_4 e^{\frac 12 \alpha_0 t} \leq c_4 e^{\alpha_0 t},
\end{align}
with $c_4=c_4(r)$. We fix this $r$ for the rest of this proof.

Taking $\xi=\epsilon\frac{x}{|x|}$ in the estimate \eqref{eq:small_est}, we get
\begin{equation}\label{psj_estimate}
	 \psj_t^r(x) \leq c_5 e^{-\epsilon|x|}e^{t\omega_r(\epsilon x/|x|)}, \quad x\in\Rd,\; t>1,
 \end{equation}
Since $\omega_r(\xi)\to 0$ as $|\xi|\to 0$,  we can choose $\epsilon>0$ such that  $\omega_r\big(\epsilon \frac x{|x|}\big) < \frac 12\alpha_0$ for every $x\in\Rd$. For this $\epsilon$ we get with \ref{th:subexp-B} that $e^{-\epsilon|x|} \leq c_6 f(|x|)$ for $|x|\geq 1$,  where $c_6$ depends on $\epsilon$.  Hence
\begin{equation}\label{psj_estimate_2}
	\psj_t^r(x) \leq c_7 e^{\frac 12 \alpha_0 t} f(|x|) \leq c_7 e^{\alpha_0 t} f(|x|), \quad |x|\geq 1,\; t > 1,
\end{equation}
for the above $\epsilon$ with a constant $c_7=c_7(r,\epsilon)$. Now we have to estimate $\psj_t^r\ast \plj_t^r$. Let
\begin{gather*}
	\psj_t^r\ast \plj_t^r(x)  
	= \left[\int_{|y|\leq 1} + \int_{|y-x|\leq 1} + \int_{|y-x|>1,|y|>1}\right] \psj_t^r(x-y) \plj_t^r(y)\, dy  
	=: \mathrm{I}_1 + \mathrm{I}_2 + \mathrm{I}_3
\end{gather*}
For $|x|>3$ and $|y|\leq 1$ we have $|y-x|>1$, and using \eqref{psj_estimate_2} and \eqref{eq:comp}, we get
\begin{align*}
	\mathrm{I}_1 
  	= \int_{|y|\leq 1} \psj_t^r(x-y) \plj_t^r(y)\, dy 
  	&\leq c_7 e^{\alpha_0 t} \int_{|y|\leq 1} f(|x-y|) \,\plj_t^r(y)\, dy \\
	&\leq c_8 e^{\alpha_0 t} f(|x|) \int \plj_t^r(y)\, dy 
	\leq c_8 e^{\alpha_0 t} f(|x|).
\end{align*}
Using the second estimate in Lemma \ref{convolutions}, \eqref{eq:comp} and \eqref{eq:choose_r}, we obtain
\begin{align*}
	\mathrm{I}_2  
	= \int_{|y-x|\leq 1} \psj_t^r(x-y) \plj_t^r(y)\, dy 
	&\leq c_9\, t  e^{-t|\nulj_r|}e^{(K(r)\vee |\nulj_r|)t} \int_{|y-x|\leq 1} \psj_t^r(x-y) f(|y|)\, dy \\
	&\leq c_{10}\, t  e^{-t|\nulj_r|}e^{(K(r)\vee |\nulj_r|)t} f(|x|) \int_{|y-x|\leq 1} \psj_t^r(x-y)\, dy \\
	&\leq c_{10}\, t  e^{K(r)t} f(|x|) \\
	&\leq c_{11} \, e^{\alpha_0 t} f(|x|).
\end{align*}
Furthermore, using \eqref{psj_estimate_2} with $\frac 12\alpha_0$, Lemma \ref{convolutions} and \eqref{eq:choose_r} (also with $\frac 12\alpha_0$), we get
\begin{align*}
	\mathrm{I}_3 
	&= \int_{|y-x|>1,|y|>1} \psj_t^r(x-y) \plj_t^r(y)\, dy \\
	&\leq c_{12}\, e^{\frac 12 \alpha_0 t} t e^{-t|\nulj_r|}e^{(K(r)\vee |\nulj_r|)t} \int_{|y-x|>1,\,|y|>1} f(|x-y|) f(|y|)\, dy \\
	&\leq c_{12}\, e^{\frac 12 \alpha_0 t} t e^{K(r)t} \int_{|y-x|>1,\,|y|>1} f(|x-y|) f(|y|)\, dy \\
	&\leq c_{13}\, e^{\alpha_0 t} \int_{|y-x|>1,\,|y|>1} f(|x-y|) f(|y|)\, dy.
\end{align*}
Using the fact that we have $K_f(1)<\infty$ under \eqref{L2}, we finally obtain 
\begin{gather*}
	\mathrm{I}_3 \leq c_{14}  e^{\alpha_0 t}f(|x|).
\end{gather*}
This completes the proof of \ref{th:subexp-C} for $t> 1$ and $|x|>3$, finishing the whole proof.
\end{proof}

\begin{lemma}\label{l-43}
	If \eqref{L1} holds, then \ref{th:subexp}.\ref{th:subexp-C} implies \ref{th:subexp}.\ref{th:subexp-B}.
\end{lemma}
\begin{proof} 
We can use \cite[Theorem 1.1]{Sztonyk2017} to see that under \eqref{L1} there are constants $c_1, c_2, c_3, c_4$, depending only on $d$ and $\nu$, such that
\begin{gather*}
	p_t(x) 
	\geq c_1 \Psi^*_{-}\left(\tfrac{1}{t}\right)^{d} e^{{-c_2|x|^2}/{t}}, 
	\quad t>c_3,\; |x|\leq c_4 t.
\end{gather*}
Therefore, \ref{th:subexp-C} shows that for every $\alpha_0>0$ there exists $C=C(\alpha_0)$ such that
\begin{gather*}
	c_1 \Psi^*_{-}\left(\tfrac{1}{t}\right)^{d} e^{{-c_2|x|^2}/{t}} \leq C e^{\alpha_0 t} f(|x|),\quad t>c_3,\; 1\leq |x| \leq c_4 t.
\end{gather*}
Fix $\epsilon>0$ and set $t=\epsilon^{-1} 4c_2|x|$ in the above inequality. We get
\begin{gather*}
	f(|x|) 
	\geq \frac{c_1}{C}  \Psi^*_{-}\left(\tfrac{\epsilon}{4c_2|x|}\right)^{d}
	e^{-|x|\left(\frac 14\epsilon +4\epsilon^{-1}c_2\alpha_0\right)},
	\quad 4\epsilon^{-1}c_2|x|>c_3,\; 1\leq |x| \leq 4\epsilon^{-1}c_2 c_4 |x|.
\end{gather*}
We may take $\alpha_0=\epsilon^2/(16c_2)$, to get
\begin{gather*}
	f(|x|) 
	\geq \frac{c_1}{C} \Psi^*_{-}\left(\tfrac{\epsilon}{4c_2|x|}\right)^{d} e^{-\frac 12\epsilon|x|},
	\quad 1\leq |x|,  \epsilon \leq \min\left\{4c_2 c_4,4c_2c_3^{-1}\right\} .
\end{gather*}
It follows from Lemma \ref{h_upper_est}	that $\Psi^*_{-}\left(\tfrac{\epsilon}{4c_2|x|}\right)^{d} \leq C_5 (4\epsilon^{-1}c_2|x|)^{1/\alpha}$, provided that $|x|\geq 1$ and $\epsilon\leq 4c_2 C_5^{\alpha}$. Thus, for every $\epsilon < 4c_2\min\left\{ c_4,c_3^{-1}, C_5^{\alpha}\right\}$, there exists a constant $c_5=c_5(\epsilon)>0$ such that $\Psi^*_{-}\left(\tfrac{\epsilon}{4c_2|x|}\right)^{d}\geq c_5 e^{-\frac 12\epsilon|x|}$ for $|x|\geq 1$, and we arrive at \ref{th:subexp-B} for such $\epsilon$, hence for all $\epsilon>0$. 
\end{proof}

\begin{lemma}\label{l-44}
	The condition \ref{th:subexp}.\ref{th:subexp-C} implies \ref{th:subexp}.\ref{th:subexp-D},  
	and \ref{th:subexp}.\ref{th:subexp-D} implies \ref{th:subexp}.\ref{th:subexp-E}.
\end{lemma}
\begin{proof} 
We begin with the implication \ref{th:subexp-C}$\Rightarrow$\ref{th:subexp-D}. Assume that \ref{th:subexp-C} holds for some $\alpha_0>0$ with the constant $C=C(\alpha_0)$. For $\alpha>\alpha_0$ we have
\begin{gather*}
	\frac{C}{\alpha - \alpha_0} f(|x|) - g_\alpha(x)
	= \int_0^\infty e^{-\alpha t} \left(C e^{\alpha_0 t}f(|x|)-p_{t}(x)\right) dt.
\end{gather*}
This shows that the function 
\begin{gather*}
	(\alpha_0,\infty) \ni \alpha \to \frac{C}{\alpha - \alpha_0} f(|x|) - g_\alpha(x)
\end{gather*}
is the Laplace transform of the non-negative measure $\mu(dt) = \left(C e^{\alpha_0 t}f(|x|)-p_{t}(x)\right) dt$, hence it is completely monotone by Bernstein's theorem. 
	
Since every completely monotone function is non-negative, \ref{th:subexp-D} implies \ref{th:subexp-E}.
\end{proof}

\begin{lemma}\label{l-45} 
	If \eqref{L1} and \eqref{L2} hold, then \ref{th:subexp}.\ref{th:subexp-E} implies \ref{th:subexp}.\ref{th:subexp-C}. 
\end{lemma}
\begin{proof} 
Fix $\alpha_0>0$. From \ref{th:subexp-E} we know 
\begin{gather*}
	\int_0^\infty e^{-\alpha s}p_s(x)\, ds 
	\leq \frac{\widetilde C}{\alpha- \frac 12 \alpha_0} f(|x|)
	\leq \frac{2\widetilde C}{\alpha_0} f(|x|) 
	\quad\text{for $\alpha \geq \alpha_0$ and $|x| \geq 1$},
\end{gather*}
for a suitable constant $\widetilde C = \widetilde C(\alpha_0/2)$. For every $t>0$ we also have
\begin{align*}
	\int_0^\infty e^{-\alpha_0 s} p_s(x)\, ds
	& \geq \int_{t}^{t+1} e^{-\alpha_0 s}p_s(x)\, ds \\
	& \geq e^{-\alpha_0 (t+1)} \int_t^{t+1} p_s(x)\, ds \\
	& = e^{-\alpha_0 (t+1)} \int_0^{1} p_{t+s}(x)\, ds \\
	& = e^{-\alpha_0 (t+1)} \int_{\Rd} p_t(x-y) \int_0^{1} p_s(y)\, ds \, dy \\
	& \geq e^{-\alpha_0 (t+1)} \inf_{|y| \leq 1/2} p_t(x-y) \int_{|y| \leq 1/2} \int_0^{1} p_s(y)\, ds \, dy,
\end{align*}
which yields
\begin{gather*}
	\inf_{|y|\leq 1/2} p_t(x-y) \leq C e^{\alpha_0 t} f(|x|), \quad t>0,\; |x| \geq 1,
\end{gather*}
where $C=\frac{2\widetilde C}{\alpha_0}e^{\alpha_0} \left(\int_{|y|\leq 1/2} \int_0^{1} p_s(y)\, ds \,dy\right)^{-1}$. Using Lemma \ref{pxy} we obtain \ref{th:subexp-C}.
\end{proof}

\section{Exponential decay} \label{sec:exp}
We will now provide the proofs of Theorems~\ref{th:exp} and~\ref{th:exp_sharp}. They will follow from a series of lemmas. Throughout this section we assume that the profile of the L\'evy density takes the form
\begin{gather}\label{eq:exp_prof_aux}
	f(r)= \exp(-\kappa r) h(r), \quad r > 0,
\end{gather}
where $\kappa>0$ and $h:(0,\infty) \to (0,\infty)$ is a decreasing function such that
\begin{gather}\label{eq:g_sub_exp_aux}
	r \mapsto \frac{\log h(r)}{r} 
	\text{\ \ is eventually increasing and\ \ } 
	\lim_{r\to \infty} \frac{\log h(r)}{r} = 0.
\end{gather}

Recall that for any fixed $\theta \in \sphere^{d-1}$ we have (see also Fig.~\ref{pic-graph-gamma})
\begin{gather*}
	\gamma_{\alpha}(\theta) 
	:= 
	\begin{cases}
		\kappa, & \alpha > \omega(\kappa \theta),\\
		\omega(\tinybullet\theta)^{-1}(\alpha), & 0 < \alpha \leq \omega(\kappa \theta).
	\end{cases}
\end{gather*} 

\begin{lemma} \label{lem:exp_1} 
	Assume \eqref{L1} and \eqref{L2}. Then there exists some $C>0$ such that for every $\alpha_0>0$ we have
	\begin{gather*}
		p_t(u\theta) 
		\leq C \exp\left(\alpha_0 t-\gamma_{\alpha_0}(\theta)\, u\right), 
		\quad u \geq 1,\; \theta \in \sphere^{d-1},\; t>0,
	\intertext{and}
  		g_{\alpha}(u\theta) 
  		\leq \frac{C}{\alpha-\alpha_0} \exp\left(-\gamma_{\alpha_0}(\theta)\, u\right), 
  		\quad u \geq 1,\; \theta \in \sphere^{d-1},\; \alpha > \alpha_0.
\end{gather*}
\end{lemma}
\begin{proof}
The second estimate follows directly from the first one by integration, so we need to estimate $p_t$ only. By \eqref{eq:small_time}, for every $\alpha_0>0$, 
\begin{gather*}
	p_t(u \theta) 
	\leq c_1 t f(u) 
	\leq c_1 \exp\left(\alpha_0 t-\gamma_{\alpha_0}(\theta)\, u\right), 
	\quad u \geq 1,\; \theta \in \sphere^{d-1},\; t\in (0,1].
\end{gather*}
Assume from now on that $t>1$. Since, by Lemma \ref{lem:exp_mom},
\begin{gather*} 
	\int_{|y| \geq 1} e^{\scalp{\xi}{y}}\nu(y) \,dy < \infty,
\end{gather*}
for every $\xi \in \Rd$ such that $|\xi| \leq \kappa$, we can use Theorem \ref{th:gen_exp} to get
\begin{gather*}
	p_t(u \theta) 
	\leq p_t(0) e^{-u\, \scalp{\xi}{\theta} + t\omega(\xi)},  
	\quad u \geq 1,\; \theta \in \sphere^{d-1},\; t >1,\; |\xi| \leq \kappa.
\end{gather*}
By \eqref{L1},
\begin{gather*}
	p_t(0) 
	= (2\pi)^{-d} \int_{\Rd} e^{-t\Psi(\xi)}\, d\xi 
	\leq (2\pi)^{-d} \int_{\Rd} e^{-\Psi(\xi)}\, d\xi < \infty, \quad t >1,
\end{gather*}
and, consequently, 
\begin{gather*}
	p_t(u \theta) 
	\leq c_2 e^{-u\, \scalp{\xi}{\theta} + t\omega(\xi)},  
	\quad u \geq 1,\; \theta \in \sphere^{d-1},\; t >1,\; |\xi| \leq \kappa,
\end{gather*}
with a uniform constant $c_2>0$. It follows directly from the definition of $\gamma_{\alpha_0}(\theta)$ that $\gamma_{\alpha_0}(\theta) \leq \kappa$ and $\omega(\gamma_{\alpha_0}(\theta)\, \theta) \leq \alpha_0$, for any $\alpha_0 >0$. Hence, taking $\xi = \gamma_{\alpha_0}(\theta)\, \theta$, we obtain 
\begin{gather*}
	p_t(u \theta) 
	\leq c_2 \exp\left(-\gamma_{\alpha_0}(\theta)\, u + \alpha_0 t\right),  
	\quad u \geq 1,\; \theta \in \sphere^{d-1},\; t >1,
\end{gather*}
which is the claimed bound. This completes the proof. 
\end{proof}

Recall that for a radial L\'evy measure $\nu$, the function $\omega$ is radial as well, and $\gamma_{\alpha_0}$ is constant on $\sphere^{d-1}$ for every $\alpha_0>0$. In that case it is convenient to set $\varpi(|\xi|) = \omega(\xi)$.

\begin{lemma} \label{lem:exp_1_lower} 
	Assume \eqref{L1}, \eqref{L2} and that the L\'evy measure $\nu$ is radial, i.e.\ $\nu(x)=\nu(|x|)$. Let $0< \alpha_0 < \varpi(\kappa)$ and $\sigma:(0,\infty) \to (0,\infty)$ be such that 
\begin{gather*} 
  \lim_{r\to\infty} \sigma(r) = 0 
	\quad\text{and}\quad 
	\lim_{r\to\infty} \frac{\sigma(r) r}{\log(1/\sigma(r))+\log r} = \infty.
\end{gather*}
Then, for every $\epsilon \in (0,1)$ there exist $\rho = \rho(\epsilon,\alpha_0)$ and $C=C(\epsilon,\alpha_0)>0$ such that 
\begin{gather} \label{eq:pt_lower}
	C e^{\alpha_0 t } 
	\leq e^{(\gamma_{\alpha_0}+\epsilon)|x|} p_t(x)  
			+ p_t(0) e^{\alpha_0 t } e^{(1+\epsilon)\left(\varpi^{\prime}(\gamma_{\alpha_0}) \frac{t}{|x|} - 1\right) \sigma(|x|)|x|}, 
	\quad |x| \geq \rho,\; t \geq \frac{|x|}{\varpi^{\prime}}(\gamma_{\alpha_0}).
\end{gather}
and
\begin{gather*}
	g_{\alpha}(x) 
	\geq \frac{C}{2(\alpha - \alpha_0)} \exp\left(-\left(\gamma_{\alpha_0}+\epsilon+\frac{\alpha-\alpha_0}{\varpi^{\prime}(\gamma_{\alpha_0})}\right)|x|\right), 
	\quad |x| \geq \rho,\; \alpha > \alpha_0.
\end{gather*}
\end{lemma}
\begin{proof} 
We first show \eqref{eq:pt_lower}. Note that $\gamma_{\alpha_0} < \kappa$. By Lemma \ref{lem:exp_mom}, we have for every $\xi \in \Rd$ such that $|\xi| \leq \kappa$
\begin{gather*}
	\int_{|y| \geq 1} e^{\scalp{\xi}{y}}\nu(y) \,dy < \infty.
\end{gather*}
Hence, by \cite[Theorem 25.17]{Sato}, for every $|x|>1$, $t>0$, $\theta \in \sphere^{d-1}$ and $\delta \in (0,1)$, 
\begin{align*}
	e^{t \varpi(\gamma_{\alpha_0})} 
	&= \int_{\Rd} e^{\gamma_{\alpha_0} \scalp{\theta}{y}}\, p_t(y) \,dy \\ 
  	&\leq \left(\int_{|y| \leq (1-\delta)|x|} + \int_{|y| \geq (1+\delta)|x|}+ \int_{(1-\delta)|x| < |y| < (1+\delta)|x|} \right) e^{\gamma_{\alpha_0}|y|}\, p_t(y) \, dy =: \mathrm{I}_1 + \mathrm{I}_2 + \mathrm{I}_3.
\end{align*}
We can pick $\rho > 1$ so large that 
\begin{gather*}
	0 < \gamma_{\alpha_0} - \sigma(r) < \gamma_{\alpha_0} + \sigma(r) \leq \kappa, \quad r \geq \rho.
\end{gather*}
Throughout the rest of the proof we assume that $|x| \geq \rho$. Take
\begin{gather*}
	\xi_1 
	= \xi_1(x,y) 
	:= \left(\gamma_{\alpha_0}  - \sigma(|x|) \right) \frac{y}{|y|}, 
	\qquad 
	\xi_2
	= \xi_2(x,y) 
	:= \left(\gamma_{\alpha_0}  + \sigma(|x|) \right) \frac{y}{|y|}, \quad y \neq 0,
\end{gather*}
and notice that $|\xi_1| \leq |\xi_2| \leq \gamma_{\alpha_0} + \sigma(|x|) \leq \kappa$. We will now use the estimate from Theorem \ref{th:gen_exp} for the integrand $p_t(y)$ appearing in the expressions for $\mathrm{I}_1$ and $\mathrm{I}_2$ with $\xi_0=\xi_1$ and $\xi_0=\xi_2$, respectively. We get
\begin{align*}
	\mathrm{I}_1 
	&\leq p_t(0)  \int_{|y| \leq (1-\delta)|x|} e^{\gamma_{\alpha_0}|y|}  e^{-\scalp{\xi_1}{y}}e^{t\varpi(|\xi_1|)}\, dy \\
    &= p_t(0) e^{t\varpi(\gamma_{\alpha_0} - \sigma(|x|))} \int_{|y| \leq (1-\delta)|x|} e^{\gamma_{\alpha_0}|y|} e^{-\gamma_{\alpha_0}|y|} e^{\sigma(|x|) \, |y|} \,dy \\
    &\leq c_1 p_t(0) e^{t\varpi(\gamma_{\alpha_0} - \sigma(|x|))}  e^{(1-\delta)\sigma(|x|)|x|+d \log|x|} \\
	&= c_1 p_t(0) e^{t\varpi(\gamma_{\alpha_0})}e^{-t(\varpi(\gamma_{\alpha_0})-\varpi(\gamma_{\alpha_0} - \sigma(|x|)))}  e^{(1-\delta)\sigma(|x|)|x|+d \log|x|} .
\end{align*}
A similar calculation, where we also perform a change of variables in one of the integrals, yields
\begin{align*}
	\mathrm{I}_2 
	&\leq p_t(0)  \int_{|y| \geq (1+\delta)|x|} e^{\gamma_{\alpha_0}|y|} e^{-\scalp{\xi_2}{y}} e^{t\varpi(|\xi_2|)} \, dy \\
    &= p_t(0) e^{t\varpi(\gamma_{\alpha_0} + \sigma(|x|))} \int_{|y| \geq (1+\delta)|x|} e^{\gamma_{\alpha_0}|y|} e^{-\gamma_{\alpha_0}|y|}e^{-\sigma(|x|) |y|} \, dy \\
    &= p_t(0) e^{t\varpi(\gamma_{\alpha_0} + \sigma(|x|))} (1+\delta)^d\sigma(|x|)^{-d}\int_{|z| \geq \sigma(|x|)|x|} e^{-(1+\delta)|z|} \, dz\\
	&\leq c_2 p_t(0) e^{t\varpi(\gamma_{\alpha_0} + \sigma(|x|))}  e^{-(1+\delta/2)\sigma(|x|)|x| + d\log (1/\sigma(x))} \\
	&= c_2 p_t(0) e^{t\varpi(\gamma_{\alpha_0})}e^{t(\varpi(\gamma_{\alpha_0} + \sigma(|x|))-\varpi(\gamma_{\alpha_0}))}  e^{-(1+\delta/2)\sigma(|x|)|x| + d\log (1/\sigma(x))}.
\end{align*}
The constants $c_1$ and $c_2= c_2(\delta)$ appearing in these estimates do not depend on $t$, $|x|>1$ or $\alpha>\alpha_0$.
Increasing $\rho$ (if need be) and using the asymptotic properties of $\sigma(x)$ as $|x| \to \infty$, we can now show that
\begin{align*}
	-t\big(\varpi(\gamma_{\alpha_0}) &-\varpi(\gamma_{\alpha_0} - \sigma(|x|))\big) +  (1-\delta)\sigma(|x|)|x|+d \log|x| \\
	&\leq -t (1-\delta/2)\varpi^{\prime}(\gamma_{\alpha_0}) \sigma(|x|) + (1-\delta/2) \sigma(|x|)|x| \\
	&\leq (1-\delta/2)\left(-\varpi^{\prime}(\gamma_{\alpha_0}) \frac{t}{|x|} + 1\right) \sigma(|x|)|x|, \quad |x| \geq \rho,
\intertext{and}
	t\big(\varpi(\gamma_{\alpha_0} &+ \sigma(|x|)) - \varpi(\gamma_{\alpha_0})\big) - (1+\delta/2)\sigma(|x|)|x| + d\log (1/\sigma(x)) \\
	&\leq t (1+\delta/4) \varpi^{\prime}(\gamma_{\alpha_0}) \sigma(|x|) - (1+\delta/4) \sigma(|x|)|x| \\
	&= (1+\delta/4)\left(\varpi^{\prime}(\gamma_{\alpha_0}) \frac{t}{|x|} - 1\right) \sigma(|x|)|x|, \quad |x| \geq \rho.
\end{align*}
In the last two inequalities we use the mean value theorem and the fact that the derivative $\varpi' = \frac d{du}\varpi$ is positive and increasing, cf.\ Fig.~\ref{pic-graph-gamma}. Since $\varpi'\big|_{(0,\kappa)}$ is continuous and $\gamma_{\alpha_0}\in (0,\kappa)$, we can find for every $\delta \in (0,1)$ some $\rho=\rho(\delta,\alpha_0)$ such that for all $|x|>\rho$ the following inequality holds
\begin{gather*}
	\left|\varpi'(\gamma_{\alpha_0}) - \varpi'(\gamma_{\alpha_0}\mp\sigma(|x|))\right| 
	\leq \frac{\delta}{4}\varpi'(\gamma_{\alpha_0}).
\end{gather*}
Consequently, for every $|x| \geq \rho$ and $t \geq \frac{|x|}{\varpi^{\prime}(\gamma_{\alpha_0})}$ we obtain
\begin{gather*}
	e^{t \varpi(\gamma_{\alpha_0})} 
	 \leq \mathrm{I}_1 + \mathrm{I}_2 + \mathrm{I}_3
	\leq c_1 p_t(0) e^{t \varpi(\gamma_{\alpha_0})} + c_2 p_t(0) e^{t \varpi(\gamma_{\alpha_0})} e^{(1+\delta/4)\left(\varpi^{\prime}(\gamma_{\alpha_0}) \frac{t}{|x|} - 1\right) \sigma(|x|)|x|} + \mathrm{I}_3.
\end{gather*}
It remains to estimate the term $\mathrm{I}_3$. Using Lemma \ref{pxy} $\lfloor 2\delta|x|+1\rfloor$ times, we get $c_3=c_3(\delta,\alpha_0)$, $c_4=c_4(\delta,\alpha_0)$ such that
\begin{align*}
	I_3 
	&= \int_{(1-\delta)|x| < |y| < (1+\delta)|x|}  e^{\gamma_{\alpha_0}|y|} p_t(y)\, dy \\
    &\leq c_3 e^{(1+2\delta)\gamma_{\alpha_0}|x|} (1/C_7)^{2\delta|x| +1} p_t(x) \\
	&\leq c_4 e^{\gamma_{\alpha_0}|x| + 2\delta(\gamma_{\alpha_0}+\log(1/C_7))|x|} p_t(x),
\end{align*}
for $\delta>0$ as above and all $|x| \geq \rho$ and $t \geq \frac{|x|}{\varpi^{\prime}(\gamma_{\alpha_0})}$.
We may choose $\delta$ so small that for any $\epsilon\in (0,1)$ we have $0 < \max\{\delta/4, 2\delta (\gamma_{\alpha_0}+\log(1/C_7))\} < \epsilon < 1$. Therefore, there exist $\rho = \rho(\epsilon,\alpha_0)$ and $c_5=c_5(\epsilon,\alpha_0)>0$ such that for every $|x| \geq \rho$ and $t \geq \frac{|x|}{\varpi^{\prime}(\gamma_{\alpha_0})}$ we have
\begin{gather*}
	c_5 e^{t\alpha_0} 
	= c_5 e^{t \varpi(\gamma_{\alpha_0})} 
	\leq e^{(\gamma_{\alpha_0}+\epsilon)|x|} p_t(x)  + p_t(0) e^{t \varpi(\gamma_{\alpha_0})} e^{(1+\epsilon)\left(\varpi^{\prime}(\gamma_{\alpha_0}) \frac{t}{|x|} - 1\right) \sigma(|x|)|x|}. 
\end{gather*}
This proves \eqref{eq:pt_lower}. 

\bigskip
We will now use \eqref{eq:pt_lower} to show the bound for the resolvent density. Let $\alpha > \alpha_0$. In view of \eqref{L1} we may assume that $\rho$ is large enough so that 
\begin{align}\label{eq:ptzero}
	p_t(0) 
	\leq \frac{c_5}{2} \frac{\alpha - \varpi(\gamma_{\alpha_0}) - (1+\epsilon)\varpi^{\prime}(\gamma_{\alpha_0}) \sigma(|x|)}{\alpha - \varpi(\gamma_{\alpha_0})}, 
	\quad\text{for}\quad 
	t \geq \frac{|x|}{\varpi^{\prime}(\gamma_{\alpha_0})} 
	\geq \frac{\rho}{\varpi^{\prime}(\gamma_{\alpha_0})}.
\end{align}
We multiply both sides of \eqref{eq:pt_lower} with $e^{-\alpha t}$ and integrate over $\frac{|x|}{\varpi^{\prime}(\gamma_{\alpha_0})} < t < \infty$. Using the estimate \eqref{eq:ptzero}, we get
\begin{align*}
	&\frac{c_5}{\alpha - \alpha_0}  e^{-\frac{\alpha-\alpha_0}{\varpi^{\prime}(\gamma_{\alpha_0})}|x|} \\
	&\leq e^{(\gamma_{\alpha_0}+\epsilon)|x|} g_{\alpha}(x)
 	+ \frac{c_5}{2} \frac{\alpha - \alpha_0 - (1+\epsilon)\varpi^{\prime}(\gamma_{\alpha_0}) \sigma(|x|)}{\alpha - \alpha_0}\frac{1}{\alpha - \alpha_0 - (1+\epsilon)\varpi^{\prime}(\gamma_{\alpha_0}) \sigma(|x|)} e^{-\frac{\alpha-\alpha_0}{\varpi^{\prime}(\gamma_{\alpha_0})}|x|}, 
\end{align*}
which can be rearranged in the following way:
\begin{gather*}
	g_{\alpha}(x) 
	\geq \frac{c_5}{2(\alpha - \alpha_0)}  e^{-(\gamma_{\alpha_0}+\epsilon)|x|} e^{-\frac{\alpha-\alpha_0}{\varpi^{\prime}(\gamma_{\alpha_0})}|x|}, \quad |x| \geq \rho.
\end{gather*}
This completes the proof. 
\end{proof}

\begin{proof}[Proof of Theorem \ref{th:exp}] 
Part \ref{th:exp-1} of the theorem is just Lemma \ref{lem:exp_1}. 

The upper bound in Part~\ref{th:exp-2} follows from the estimate in Part~\ref{th:exp-1} and the continuity of the map $\alpha \mapsto \gamma_{\alpha}$ in the following way: fix $\alpha >0$ and $\epsilon>0$ and choose $\delta >0$ so small that $\gamma_{\alpha-\delta} \geq \gamma_{\alpha} - \epsilon$; set $\alpha_0 = \alpha - \delta$.

The lower estimate in Part~\ref{th:exp-2} is a consequence of Lemma \ref{lem:exp_1_lower}. 
\emph{Case 1:} $\alpha < \varpi(\kappa)$. Since the map $(0, \varpi(\kappa)) \ni \alpha \mapsto \varpi^{\prime}(\gamma_{\alpha})$ is continuous (it may become infinite at $\varpi(\kappa)$, e.g.\ in the relativistic case, see Fig.~\ref{pic-graph-gamma}), we can find for any $\epsilon >0$ a $\delta = \delta(\epsilon) > 0$ such that $\delta \leq \frac 12 \epsilon \varpi^{\prime}(\gamma_{\alpha-\delta})$. We can now use the lower bound for the resolvent density from Lemma \ref{lem:exp_1_lower} with $\frac 12\epsilon$ and $\alpha_0 = \alpha - \delta$. Since $\gamma_{\alpha-\delta} \leq \gamma_{\alpha}$, we get the assertion for $|x| \geq \rho$, for some $\rho= \rho(\epsilon,\alpha)$. 
For $1 \leq |x| \leq \rho$, however, the claimed bound is trivially true, as the functions appearing on either side are bounded above and bounded away from zero. 

\emph{Case 2:} $\alpha\geq \varpi(\kappa)$. In this case we use \eqref{eq:res_low} to get $g_\alpha(x)\geq c f(|x|)$. Since $f(|x|)=e^{-\kappa|x|}h(|x|)$ and, since $h$ is subexponential, we conclude that for all $|x|$ and any $\epsilon >0$ there exists $c = c(\epsilon)$ such that  $f(|x|) \geq c e^{-(\kappa+\epsilon)|x|}$.
\end{proof}

For the proof of Theorem~\ref{th:exp_sharp} let us recall that
\begin{gather*}
	\omega^*(\kappa):= \sup_{\theta \in \sphere^{d-1}} \omega(\kappa \theta). 
\end{gather*}

\begin{proof}[Proof of Theorem~\ref{th:exp_sharp}] 
Since the estimate for the resolvent density $g_\alpha$ follows from the heat-kernel bound by integration, it is enough to estimate $p_t$ only. From \eqref{eq:small_time} we get
\begin{gather*}
	p_t(x) 
	\leq c_1 t f(|x|) \leq c_1 e^{\alpha_0 t} f(|x|), 
	\quad |x| \geq 1, \;t\in (0,1].
\end{gather*}
Therefore, we can assume in the remaining part of the proof that $t>1$. By \eqref{eq:small_est},
\begin{gather*}
	 \psj_t^r(x) \leq c_2 e^{-\scalp{\xi}{x}+t\omega_r(\xi)}, \quad x,\, \xi\in\Rd,\; r \geq 1,
\intertext{where}
	\omega_r(\xi) 
	= \int \left(\cosh (\scalp{\xi}{y}) - 1\right) \nusj_r(dy) 
	= \int_{|y|<r} \left(\cosh (\scalp{\xi}{y}) - 1\right) \nu(dy).
\end{gather*}
Using this estimate and the decomposition \eqref{eq:decomp}, we have for every $r, R >1$, $|x| \geq 3R$ and $\xi_1, \xi_2 \in \Rd$ 
\begin{align*}
	p_t(x) 
	&= e^{-t|\nulj_r|} \psj_t^r(x) + \psj_t^r\ast \plj_t^r(x) \\
    &\leq \psj_t^r(x) + \left(\int_{|y| \leq 1} + \int_{\substack{|y| > 1\\ |y-x| > R}} + \int_{|y-x| \leq R} \right) \psj_t^r (x-y) \plj_t^r(y)\,dy \\
	&\leq  c_2 \left(e^{-\scalp{\xi_1}{x} + t\omega_r(\xi_1)} + e^{t\omega_r(\xi_1)} \int_{|y| \leq 1} e^{-\scalp{\xi_1}{(x-y)}}\plj_t^r(y)\,dy
	+ \int_{\substack{|y| > 1\\ |y-x| > R}} e^{t\omega_r(\xi_2)}  e^{-\scalp{\xi_2}{(x-y)}}\plj_t^r(y)\,dy\right) \\
	&\hphantom{\qquad c_2 \left(e^{t\omega_r(\xi_1)} \int_{|y| \leq 1} e^{-\scalp{\xi_1}{(x-y)}}\plj_t^r(y)\,dy
		+ \int_{\substack{|y| > 1\\ |y-x| > R}} e^{t\omega_r(\xi_2)}  e^{-\scalp{\xi_2}{(x-y)}}\plj_t^r(y)\,dy\right)}
	\mbox{} + \sup_{|y-x| \leq R} \plj_t^r(y) \\
	& =: c_2 \left(\mathrm{I}_1 + \mathrm{I}_2 + \mathrm{I}_3\right) + \mathrm{I}_4.
\end{align*}
We now take specific values for $\xi_1, \xi_2$:
\begin{align*}
	\xi_1 &= a \frac{x}{|x|} &\text{where}&& a = a(x) &= \kappa-\frac{\log h(|x|)}{|x|},\\
	\xi_2 &= b \frac{x-y}{|x-y|}&\text{where}&& b = b(x,y) &= \kappa-\frac{\log h(|x-y|)}{|x-y|}.
\end{align*}
Denote
\begin{gather*}
	\delta_R := \sup_{|z| \geq R} \frac{|\log h(|z|)|}{|z|} < \infty
\intertext{and}
	\omega_r^*(u) := \sup_{|\xi|=u} \omega_r(\xi), 
	\qquad 
	\omega_r^*(u,R) := \sup_{u-\delta_R \leq |\xi| \leq u+\delta_R} \omega_r(\xi), \quad u >0.
\end{gather*}
Then we get
\begin{gather*}
	\mathrm{I}_1 + \mathrm{I}_2 
	= e^{t\omega_r\left(a \frac{x}{|x|}\right)} f(|x|) \left(1+ \int_{|y| \leq 1} e^{a \frac{\scalp{x}{y}}{|x|}}\plj_t^r(y) \,dy\right) 
    \leq e^{t\omega_r^*(\kappa,R)} f(|x|) \left(1+ e^{\kappa+\delta_R}\right) 
\intertext{and}
	\mathrm{I}_3 
	\leq \int_{\substack{|y| > 1\\ |y-x| > R}}e^{t\omega_r\left(b \frac{x-y}{|x-y|}\right)} f(|x-y|)\plj_t^r(y)\,dy 
	\leq e^{t\omega_r^*(\kappa,R)} \int_{\substack{|y| > 1\\ |y-x| > 1}} f(|x-y|)\plj_t^r(y) \,dy.
\end{gather*}
Using Lemma \ref{convolutions}, the fact that \eqref{L2} implies $K_f(1) < \infty$, and \eqref{eq:comp}, we further obtain
\begin{gather*}
	\mathrm{I}_3 
	\leq c_3 t e^{t(K(r)+\omega_r^*(\kappa,R))} \int_{\substack{|y| > 1\\ |y-x| > 1}} f(|x-y|)f(|y|) \,dy
    \leq c_4 t e^{t(K(r)+\omega_r^*(\kappa,R))}  f(|x|)
\intertext{and}
	\mathrm{I}_4 
	\leq c_5 t e^{t K(r)} \sup_{|y-x| \leq R} f(|y|) 
	\leq c_6 t e^{t K(r)} f(|x|),
\end{gather*}
with $c_4=c_4(r)$ and $c_6=c_6(r,R)$. 

If we collect all of the above estimates, we see
\begin{gather*}
	p_t(x) 
	\leq c_7 t e^{t(K(r)+\omega_r^*(\kappa,R))}  f(|x|),
	\quad  r, R >1,\; |x| \geq 3R,\; t>1,
\end{gather*}
with $c_7=c_7(r,R)$. 

In order to get the claimed estimate, we fix $\alpha_0>\omega^*(\kappa)$. By \eqref{L2} we can choose $r>1$ large enough so that
\begin{gather*}
	K(r) \leq \frac{\alpha_0-\omega^*(\kappa)}{3}.
\end{gather*}
Because of the continuity of the map $\xi \mapsto \omega_r(\xi)$ (with $r$ as before) and the assumption \eqref{eq:g_sub_exp_aux} on the profile $h$, we can find $R>1$ so large that 
\begin{gather*}
	\omega_r^*(\kappa,R) \leq \omega_r^*(\kappa) + \frac{\alpha_0-\omega^*(\kappa)}{3}.
\end{gather*}
In particular, as $\omega_r(\kappa \theta) \leq \omega(\kappa \theta) \leq \omega^*(\kappa)$ for every $\theta \in \sphere^{d-1}$, 
\begin{gather*}
	\omega_r^*(\kappa,R) \leq \omega^*(\kappa) + \frac{\alpha_0-\omega^*(\kappa)}{3}.
\end{gather*}
Furthermore, we can easily find a constant $c_8>0$ such that $t \leq c_8 \exp\left(\frac{\alpha_0-\omega^*(\kappa)}{3} t\right)$, $t > 1$. Therefore, we conclude that there is some $c_9>0$ such that 
\begin{gather*}
	p_t(x) 
	\leq c_9 e^{t(3\frac{\alpha_0-\omega^*(\kappa)}{3} + \omega^*(\kappa) )}  f(|x|) 
	= c_9 e^{\alpha_0 t} f(|x|),
	\quad |x| \geq \rho,\; t>1,
\end{gather*}
with $\rho := 3R$. This completes the proof of Theorem~\ref{th:exp_sharp}.
\end{proof}

\section{Decay of the bound states} \label{sec:bound-states}

In this section we prove Corollaries~\ref{cor:bound_state} and~\ref{cor:bound_state_exp}. The corresponding proofs in the papers \cite{Carmona-Masters-Simon, Kaleta-Lorinczi} are probabilistic -- they are based on the Feynman--Kac formula and, therefore, they are restricted to the potentials from the Kato class for the operator $L$, see \cite{Demuth-Casteren}. Here we use an analytic approach, which is based on quadratic forms and the perturbation formula. This proof works well for negative potentials, even if they are more singular. 

Throughout this section it is convenient to write $e^{tL}$ instead of $P_t$. Recall that the quadratic form related to the operator $-L$ is defined as
\begin{gather*} 
	\Ecal^{(-L)}(u,v) := \lim_{t\searrow 0} \Ecal^{(-L)}_t (u,v), \quad u, v \in \Dcal({\Ecal^{(-L)}}),
\end{gather*} 
where
\begin{gather*}
  \Ecal^{(-L)}_t (u,v) := \frac{1}{t} \left\langle u-e^{tL} u,v \right\rangle 
  \quad\text{and}\quad 
  \Dcal(\Ecal^{(-L)}) := \left\{u\in L^2(\Rd):\: \lim_{t\searrow 0} \Ecal^{(-L)}_t(u,u) < \infty \right\}, 
\end{gather*} 
see \cite{Fukushima} or \cite[Chapter 6]{Casteren}. It is a simple consequence of the spectral representation of the operators $e^{tL}$, $t>0$, that the map
\begin{align} \label{eq:decr}
	(0,\infty) \ni t \mapsto \Ecal^{(-L)}_t (u,u) \text{\ is decreasing for every\ } u \in L^2(\Rd).
\end{align}
Therefore, the (improper) limit $\lim_{t\searrow 0} \Ecal^{(-L)}_t (u,u) \in [0,\infty]$ always exists for all $u \in L^2(\Rd)$. Recall that the form $\left(\Ecal^H, \Dcal(\Ecal^H)\right)$ of the Schr\"odinger operator $H = -L+V$ is given by \eqref{eq:eq_forms}; in particular, $\Dcal(\Ecal^H) = \Dcal(\Ecal^{(-L)})$. Clearly, $\Ecal^H$ can also be represented in terms of the semigroup $\left\{e^{-tH}, t  \geq 0\right\}$ in a similar way as above.

By a general argument from spectral theory, we have for $u, v \in L^2(\Rd)$ that $e^{-t H} u \in \Dcal(\Ecal^{H})$ and $e^{t L} v \in \Dcal(\Ecal^{(-L)})$, $t>0$, cf.\ \cite[Lemma 1.3.3 (i)]{Fukushima}. Thus,
\begin{gather*}
	\frac{d}{ds} \left\langle e^{-(t-s) H} u, e^{s L} v \right\rangle 
	= \Ecal^{H} (e^{-(t-s) H} u,e^{s L} v) - \Ecal^{(-L)} (e^{-(t-s) H} u,e^{s L} v) 
\end{gather*}
for all $u, v \in L^2(\Rd)$ and $0<s<t$. Integrating both sides of this equality and using \eqref{eq:eq_forms}, we get 
\begin{align}\label{eq:weak_perturb_form}
	\left\langle u,  e^{t L} v \right\rangle - \left\langle e^{-t H} u, v \right\rangle 
	=  \int_0^t \left\langle  V e^{-(t-s) H} u,  e^{s L} v \right\rangle ds, 
	\quad u, v \in L^2(\Rd),\; t>0.
\end{align}

\begin{proof}[Proof of Corollaries \ref{cor:bound_state} and \ref{cor:bound_state_exp}]
Let $\phi \in L^2(\Rd)$ be an eigenfunction of $H$ with eigenvalue $\lambda<0$ as in \eqref{eq:eigenequation}. Use \eqref{eq:weak_perturb_form} with $u = \phi$ and an arbitrary $v \in \Dcal(\Ecal^{(-L)})$. Because of the self-adjointness of $e^{t L}$ and the eigenequation $e^{-t H} \phi = e^{-\lambda s} \phi$, we have
\begin{gather}\label{eq:perturb_eigeneq}
	\left\langle \phi,  v \right\rangle 
	= e^{\lambda  t} \left\langle  e^{t L} \phi,v \right\rangle - \int_0^t e^{\lambda s} \left\langle V \phi, e^{s L} v \right\rangle\, ds, \quad t>0.
\end{gather}
By the Cauchy--Schwarz inequality, 
\begin{gather*}
	|\left\langle |V| |\phi|, e^{s L}  v \right\rangle|^2 
	\leq  \left\langle |V| |\phi|, |\phi| \right\rangle \left\langle |V| e^{s L}  v, e^{s L}  v \right\rangle
\end{gather*}
and by \eqref{eq:relative_bdd} and \cite[Lemma 1.3.3 (i)]{Fukushima}, we can further write
\begin{gather*}
	\left\langle |V| |\phi|, |\phi| \right\rangle 
	\leq a \Ecal^{(-L)}(|\phi|,|\phi|) + b \left\|\phi\right\|_{L^2}^2 < \infty,\\
	\left\langle |V| e^{s L} v, e^{s L} v \right\rangle 
	\leq a \Ecal^L(e^{s L} v,e^{s L} v) + b \left\|e^{s L} v \right\|_{L^2}^2 
	\leq a \Ecal^L(v,v) + b \left\| v \right\|_{L^2}^2 < \infty, 
	\quad s >0.
\end{gather*}
Moreover, $0 \leq e^{\lambda  t} |\left\langle  e^{t L} \phi,v \right\rangle| \leq e^{\lambda  t} \left\|\phi \right\|_{L^2} \left\|v \right\|_{L^2} \to 0$ as $t \to \infty$. Therefore, letting $t \to \infty$ in \eqref{eq:perturb_eigeneq}, using Fubini's theorem and the fact  that $|V|=-V$, reveals
\begin{gather*}
	\left\langle \phi,  v \right\rangle 
	=  \left\langle  \int_0^{\infty} e^{-\alpha s} e^{s L} (|V| \phi) \,ds,  v \right\rangle, 
	\quad v \in \Dcal(\Ecal^{(-L)}),
\end{gather*}
where $\alpha:= |\lambda| = - \lambda>0$. In particular, we have for Lebesgue almost all $x\in\Rd$
\begin{align}\label{eq:eig_aux_est}
	|\phi(x)| &\leq \int_{\Rd} g_{\alpha} (x-z) |V(z)| |\phi(z)| \,dz \qquad \text{(Lebesgue a.e.)},
\intertext{and, if $\phi$ is the ground state,}\label{eq:gs_aux_est}
	\phi(x) &= \int_{\Rd} g_{\alpha} (x-z) |V(z)| \phi(z)\, dz \qquad \text{(Lebesgue a.e.)}.
\end{align}
Since $V \leq 0$, the semigroup operators $e^{-t H}$ improve positivity (this property is inherited from $e^{t L}$, since \eqref{L1} guarantees the existence of transition densities) and the ground state is strictly positive, see \cite[Theorem XIII.44]{Reed-Simon}. By \eqref{H2}, $\phi$ is continuous on the set $\{z:|z|>r\}$. By \eqref{H2} and \eqref{eq:ass_potential}, the map $x \mapsto \int_{|z|>r} g_{\alpha} (x-z) |V(z)| |\phi(z)| \,dz$ is also continuous on $\{x:|x|>r\}$, as it is the convolution of an $L^{\infty}$-function and an $L^1$-function \cite[Theorem 14.8 (ii)]{Schilling2}. Using the continuity of the function $x \mapsto p_t(x)$, for every fixed $t>0$, \eqref{eq:sup_pt} and \eqref{eq:small_time}, it is clear that $\Rd \setminus \{0\} \ni x \mapsto g_{\alpha}(x)$ is again continuous. Further, let $0 \leq \psi \in C_c^{\infty}(\Rd) \subset \Dcal(\Ecal^L)$ be such that $\I_{B_r(0)} \leq \psi$. By the Cauchy--Schwarz inequality and \eqref{eq:relative_bdd},
\begin{align} \label{eq:loc_int}
	\int_{|y| < r} |V(y) \phi(y)| \,dy 
	\leq \left\langle |V| |\phi|, \psi \right\rangle 
	\leq  \sqrt{\left\langle |V| |\phi|, |\phi| \right\rangle \left\langle |V| \psi, \psi \right\rangle} 
	< \infty.
\end{align}
Thus, for $|x| \geq 2r$,
\begin{gather*}
	\int_{|z| < r} g_{\alpha} (x-z) |V(z)| |\phi(z)| \, dz 
	\leq \sup_{|z| \geq r} g_{\alpha}(z) \int_{|z| < r} |V(z)| |\phi(z)| \,dz < \infty,
\end{gather*}
and we conclude with dominated convergence that $x \mapsto \int_{|z| < r} g_{\alpha} (x-z) |V(z)| |\phi(z)| \,dz$ is continuous on the set $\{x:|x|>2r\}$.

From the above discussion we see that the functions on both sides of \eqref{eq:eig_aux_est} and \eqref{eq:gs_aux_est} are continuous on $\{x:|x|>2r\}$, which allows us to understand the (in)equalities pointwise for all $|x|>2r$.

In view of \eqref{H2} and \eqref{eq:ass_potential} we can now follow the proof of \cite[Lemma 4.4]{Jakubowski-Kaleta-Szczypkowski} to show that for every $\delta \in \left(0,\frac 12\alpha\right)$ there exists some $\tilde r = \tilde r(\delta) > 2r$, such that
\begin{gather*}
	|\phi(x)| 
	\leq \int_{|z| \leq \tilde r} g_{\alpha-\delta} (z-x) |V(z)| |\phi(z)| \,dz, 
	\quad |x| \geq \tilde r.
\end{gather*} 
We use this inequality to establish the upper estimates in Corollary~\ref{cor:bound_state} and~\ref{cor:bound_state_exp}. First we consider Corollary \ref{cor:bound_state}.\ref{cor:bound_state-a}. By Theorem \ref{th:subexp}.\ref{th:subexp-E} (e.g.\ with $\alpha_0=\alpha-2\delta$) combined with the uniform comparability property \eqref{eq:comp}, \eqref{eq:loc_int}, \eqref{eq:ass_potential} and \eqref{H2}, we get
\begin{gather*}
	|\phi(x)| 
	\leq c_1 \int_{|z| \leq \tilde r} f(|z-x|) |V(z)| |\phi(z)| \,dz 
    \leq c_2 \left(\int_{|z| \leq \tilde r} |V(z)| |\phi(z)| \,dz \right) f(|x|), 
    \quad |x| \geq 3 \tilde r =: \rho.
\end{gather*} 
This gives the upper bound in Corollary \ref{cor:bound_state}.\ref{cor:bound_state-b}.

In order to show the upper bound in Corollary \ref{cor:bound_state_exp}.\ref{cor:bound_state_exp-a}, \ref{cor:bound_state_exp-b}, we assume that $\alpha > \omega^*(\kappa)$. Pick $\delta \in \left(0,\frac 12\alpha\right)$ in such a way that $\alpha-\delta > \omega^*(\kappa)$, and proceed in the same way as in the proof of Corollary~\ref{cor:bound_state}.\ref{cor:bound_state-a}, using Theorem \ref{th:exp_sharp} instead of Theorem \ref{th:subexp}.\ref{th:subexp-E}. 

We will finally establish the upper estimates in Corollary \ref{cor:bound_state_exp}.\ref{cor:bound_state_exp-c}, \ref{cor:bound_state_exp-d}. Fix $\epsilon>0$, assume that $\alpha \leq \omega(\kappa)$ and use the upper bound in Theorem~\ref{th:exp}.\ref{th:exp-2} to get
\begin{align*}
	|\phi(x)| 
	&\leq c_3 \int_{|z| \leq \tilde r} \exp\left(-(\gamma_{\alpha-\delta}-\epsilon/2)|x-z|\right) |V(z)| |\phi(z)| \,dz  \\
    &\leq c_4 \left(\int_{|z| \leq \tilde r} |V(z)| |\phi(z)| \,dz \right) \exp\left(-(\gamma_{\alpha-\delta}-\epsilon/2)|x|\right), 
    \quad |x| \geq 3 \tilde r =: \rho.
\end{align*}
Because of the continuity of the map $\alpha \mapsto \gamma_{\alpha}$ we know that for every $\alpha>0$ and $\epsilon >0$ there exists some $\delta >0$ such that $\gamma_{\alpha-\delta} \geq \gamma_{\alpha} - \epsilon/2$. This gives the claimed estimates.  

In order to prove the lower estimates for the ground state in Corollaries \ref{cor:bound_state}.\ref{cor:bound_state-b} and \ref{cor:bound_state_exp}.\ref{cor:bound_state_exp-b}, \ref{cor:bound_state_exp-d}, we proceed in a similar way. We use the inequality
\begin{gather*}
	\phi(x) 
	\geq \int_{|z| \leq \tilde r} g_{\alpha} (z-x) |V(z)| \phi(z) \,dz, 
	\quad |x| \geq \tilde r,
\end{gather*}
and apply \eqref{eq:res_low} and the lower bounds in Theorems \ref{th:exp}.\ref{th:exp-2} and \ref{th:exp_sharp} combined with \eqref{eq:comp}. Here we can increase $\tilde r$ if necessary as we want to be sure that $\int_{|z| \leq \tilde r} |V(z)| \phi(z) \,dz >0$. 
\end{proof}

\bibliographystyle{abbrv}

\end{document}